\documentclass[11pt,reqno, draft]{amsart}

\usepackage{commons}

\usepackage{setspace}
\usepackage{indentfirst}

\newtheorem{theorem}{Theorem}[section]
\newtheorem{lemma}{Lemma}[section]
\newtheorem{proposition}{Proposition}[section]

\newtheoremstyle{intro}
  {12pt}
  {6pt}
  {\itshape}
  {}
  {\bfseries}
  {.}
  {.5em}
  {}

\theoremstyle{intro}
\newtheorem{introthm}{Theorem}

\newtheorem{intropro}{Problem}

\theoremstyle{remark}
\newtheorem{remark}{Remark}[section]
\newtheorem*{claim*}{Claim}

\theoremstyle{definition}
\newtheorem{example}{Example}[section]
\newtheorem{definition}{Definition}[section]

\title[Symplectic realizations of holomorphic Poisson manifolds]
{Symplectic realizations of holomorphic Poisson manifolds}
\thanks{Research partially supported by the National Science
 Foundation grants  DMS1101827,  DMS1406668,  DMS-1707545, and DMS-200159.}

\author{Damien Broka}
\address{Department of Mathematics, Penn State University}
\email{damien.broka@psu.edu} 
\author{Ping Xu}
\address{Department of Mathematics, Penn State University}
\email{ping@math.psu.edu}

\begin{document}

\renewcommand{\lin}[1]{\ell_{#1}}
\newcommand{\ham}[1]{\cH(#1)}
\newcommand{\LLie}{\mathbb{L}\mathrm{ie}}
\newcommand{\pomega}{\underline\omega} 
\newcommand{\ppi}{\underline\pi} 
\newcommand{\pN}{\underline{N}} 
\newcommand{\eq}[1]{Eq.\ (\ref{#1})}
\newcommand{\uomega}{\pomega} 
\newcommand{\uJ}{\underline J} 
\newcommand{\bJ}{\bar J} 
\newcommand{\pJ}{\uJ} 
\newcommand{\had}{\widehat{\ad}} 
\renewcommand{\cD}{\DD} 
\newcommand{\tP}{\tilde P} 
\renewcommand{\P}{P} 
\newcommand{\Pl}{\P_{\text{loc}}} 
\newcommand{\bPl}{\bar\P_{\text{loc}}} 
\newcommand{\rgpd}{\Sigma} 
\newcommand{\bP}{\bar{P}} 
\newcommand{\bN}{{\bar{N}}} 
\newcommand{\tN}{\bN}
\newcommand{\bpi}{\bar\pi} 
\renewcommand{\bemol}{^\flat} 
\newcommand{\tomega}{\tilde{\omega}} 
\renewcommand{\bomega}{\overline{\omega}} 
\newcommand{\can}{\text{can}} 
\newcommand{\diff}{d} 
\newcommand{\flippy}{j} 
\newcommand{\pairy}{i} 
\newcommand{\flippypairy}{\flippy'} 
\newcommand{\tpi}{\tilde\pi} 

\renewcommand{\transpose}[1]{\ensuremath{#1^{\scriptscriptstyle T}}}

\newcommand{\hf}[1]{\mathcal{O}_{#1}} 
\newcommand{\cotalg}[1]{ (T^*#1)_{\pi}}
\newcommand{\longtalg}[1]{(T#1, \text{pr}: T#1 \to #1, [\cdot, \cdot]_{\text{Lie}})}
\newcommand{\talg}[1]{T#1}

\newcommand{\holom}[1]{\ensuremath{\mathscr{#1}}}
\renewcommand{\smooth}[1]{#1}
\renewcommand{\smalcirc}{\circ}

\newcommand\numberthis{\addtocounter{equation}{1}\tag{\theequation}}

\newcommand{\rood}[1]{\textcolor{red}{#1}}
\newcommand{\blauw}[1]{\textcolor{blue}{#1}}
\newcommand{\oranje}[1]{\textcolor{orange}{#1}}
\newcommand{\Liealgebroid}{Chevalley–Eilenberg  differential }

\begin{abstract}
Symplectic realization  is a longstanding problem which can be traced back to
Sophus Lie. In this paper, we present an explicit solution to this problem for
an arbitrary holomorphic Poisson manifold. More precisely, 
for any holomorphic Poisson manifold $(\holom{X}, \pi)$ with underlying real smooth manifold $X$, we prove that there exists a 
holomorphic symplectic  structure in a neighborhood $Y$ of the zero section
of $T^*X$ such that the projection map is a holomorphic
symplectic realization of the given holomorphic Poisson manifold, and moreover  the
zero section is a holomorphic Lagrangian submanifold.
We describe an explicit construction for such  a new holomorphic  symplectic
 structure on $Y \subseteq T^*X$.
\end{abstract}

\maketitle

\section{Introduction}

The notion  of  ``symplectic realizations"
 can   be   traced back to  Sophus Lie who used  the name  ``function group".
In \cite{MR1510035}, Lie  defined a ``function group" as a collection
of functions of the canonical variables $(q_{1}, \cdots , q_{n},
p_{1}, \cdots , p_{n})$ which is a subalgebra under the
canonical Poisson bracket and  generated by a finite
number  of independent functions $\phi_{1}, \cdots , \phi_{r}$.
In  modern language, this means that $\RR^{r}$ has a
Poisson structure induced from the canonical
symplectic structure $\RR^{2n}$ in the sense
that $\Phi= (\phi_{1}, \cdots , \phi_{r}): \RR^{2n} \lon \RR^{r}$
is a Poisson map. In the $C^\infty$-context,  a symplectic realization  of 
a Poisson manifold $M$,   as defined by Weinstein \cite{MR723816}
(called a  full  symplectic realization), is a Poisson map from a 
symplectic manifold $V$ to $M$ which is a surjective submersion.
Since  Sophus Lie's treasure work
on the theory of  transformation group \cite{MR1510035}, the following
has become a central question:

\begin{intropro}\label{problem-a}
Does a symplectic realization always exist for a given Poisson manifold?
\end{intropro}

In fact, this question is closely related to Lie's theory on Lie groups.
To get a flavor of this, consider the  Lie-Poisson manifold $\dualstar \frakg$
corresponding to a finite dimensional
 Lie algebra $\frakg$. A natural choice of a  
symplectic realization is $\Phi : T^*G \to \dualstar \frakg $ with the canonical
cotangent  symplectic structure on $T^*G$ and $\Phi$ being the
left translation, where    $G$ is  a Lie group   with 
Lie algebra $\frakg$, and $\dualstar \frakg\cong T_e^*G$. 
Lie himself  proved that  a symplectic realization always exists locally for
 any smooth  Poisson manifold of constant rank \cite{MR1510035}.
A local  existence  theorem for symplectic realizations of general
smooth Poisson manifolds, was proved by Weinstein  in 1983 \cite{MR723816}.
Subsequently, Karasev \cite{MR854594} and    Weinstein \cite{MR866024} 
independently proved the global existence  theorem by  gluing methods. 
Indeed,
 they proved a stronger result: for any 
	 $C^\infty$-Poisson manifold,   there exists
 an essentially unique, distinguished, symplectic realization which possesses
a compatible local {\em groupoid} structure \cite{MR866024}, a device which is
 now altogether called  a symplectic local groupoid \cite{MR996653}. 
 Furthermore,
 the infinitesimal object corresponding to this local Lie groupoid -- its so-called Lie algebroid, as introduced by Pradines
\cite{MR0216409} --  
can be proved \cite{MR996653} to be isomorphic to the cotangent Lie algebroid 
$\cotalg{M}$ canonically associated to the Poisson
manifold $(M, \pi)$. The  bracket of this 
Lie algebroid essentially extends the natural
Lie bracket relation on exact forms:
 $[df, dg]_* = d\{ f, g\}$  in an obvious way.
For more details on symplectic local groupoids, see \cite{MR996653, MR1214142, MR866024}.
\color{black}

For a given Poisson manifold $(M, \pi)$,  the pair of Lie algebroids
 $(\cotalg{M}, \talg{M})$, where $\talg{M}$ is the standard tangent
Lie  algebroid of $M$, constitute an example of the so-called
{\em Lie bialgebroids} \cite{MR1262213}.  From the theory of integration of
 Lie bialgebroids of Mackenzie--Xu \cite{MR1746902} (which
  extends  the  classical theory of Drinfeld 
\cite{MR688240, MR934283} for  integrating Lie bialgebras), it follows that, under some mild topological assumption,  a Lie groupoid
with Lie algebroid $\cotalg{M}$ automatically carries a compatible symplectic
 structure, and is therefore a symplectic groupoid. As a consequence,
 any \emph{local} Lie  groupoid with Lie algebroid $\cotalg{M}$ -- the existence of which 
 is guaranteed \cite{MR0216409} -- gives {\em automatically}
 a symplectic realization of the underlying Poisson manifold.
In this  way, the Mackenzie--Xu integration method provided
 an alternative proof of  the
existence of global symplectic realizations \cite{MR1746902}. However, 
all these results are existence results and are not constructive.

In 2001, while investigating Poisson sigma models, 
Cattaneo--Felder \cite{MR1938552} discovered an explicit construction 
for the symplectic groupoid of an integrable Poisson manifold. 
Over the past 20 years, this construction --- 
a certain quotient space of the Banach manifold 
of what would eventually be recognized in~\cite{MR1973056} 
as $A$-paths in the cotangent Lie algebroid of $M$ --- 
inspired many important works in Poisson geometry,
among which the solution to the problem of integrability 
of Lie algebroids \cite{MR1973056}. 
Although local symplectic groupoids are not mentioned explicitly in~\cite{MR1938552},
an explicit construction of local symplectic groupoids is essentially given 
in~\cite[Theorem 4.7]{MR1938552} since the hypothesis 
\cite[Assumption 4.6]{MR1938552} actually holds in a neighborhood of the unit space. 
We refer the interested reader to \cite{MR2063018} for more details
on the relation between the approach to the integration of general Lie
algebroids developed in \cite{MR1973056} and the results on the integration 
of Poisson manifolds exposed in~\cite{MR1938552}.  
In particular, it is shown in~\cite{MR2063018} that, 
provided integrability is assumed, 
the integration construction for Lie algebroids can be seen 
as a particular case of the integration construction for Poisson manifolds.
Furthermore, around the same time, \v{S}evera observed independently 
that the approach of \cite{MR1938552} can be generalized to all Lie algebroids 
\cite{severa2001title}.
The main novelty in~\cite{MR1973056} is the precise integrability criterion 
guaranteeing the existence of global smooth Lie groupoids.
\color{black}

Although a lot of works focus on symplectic realizations
in the $C^\infty$-context,  very little exists and is known  in the
 holomorphic  context.
 A holomorphic Poisson manifold is a complex manifold $\holom{X}$ whose sheaf
of holomorphic functions $\hf{\holom{X}}$ is a sheaf of Poisson algebras.
Symplectic realizations can be defined in a similar fashion as
in the $C^\infty$-case. Thus a natural question  is

\begin{intropro}\label{problem-b}
  Does a symplectic realization always exist for a given holomorphic 
Poisson manifold? And, if so, is it possible to describe an explicit  construction of a certain class of distinguished ones?
\end{intropro}

To  any  holomorphic Poisson manifold $(\holom{X}, \pi)$ with underlying real smooth manifold $\smooth{X}$,
one associates two $C^\infty$-Poisson 
 bivector fields.
To see this,  write the holomorphic  Poisson tensor
$\pi \in \Gamma(\wedge^2 T^{1,0}\smooth{X})$ as $\pi_R + i\pi_I $,  where
 $\pi_R$ and $\pi_I \in \Gamma(\wedge^2 T \smooth{X})$ are bivector fields. Then
both $\pi_R$ and $ \pi_I$ are indeed  $C^\infty$-Poisson
 bivector fields \cite{MR2439547}.
In 2009, Laurent-Gengoux, Sti\'enon and Xu proved that a holomorphic
 Poisson manifold is integrable if and only if
 either  $(X, \pi_R)$, or $(X, \pi_I)$ are integrable
 as a real $C^\infty$-Poisson manifold (Theorem 3.22 \cite{MR2545872}).
Since any  $C^\infty$-Poisson manifold admits a symplectic local groupoid,
as a consequence,  this result  of Laurent-Gengoux, Sti\'enon and Xu
implies that symplectic realizations do exist for any 
holomorphic Poisson manifolds. However, 
the conclusion is again not constructive. The purpose of the  present
paper is to describe an explicit construction of 
 such a holomorphic symplectic local groupoid
based on
the  Cattaneo-Felder's Poisson sigma model approach \cite{MR1938552},
and therefore to give an explicit affirmative answer to Problem \ref{problem-b}.

Our approach is based on the observation that a holomorphic Poisson manifold
$(\holom{X}, \pi)$, where $\pi =\pi_R+i \pi_I$,  gives rise to a
 Poisson--Nijenhuis \cite{MR1077465, MR773513} 
structure $(X, \pi_I, J)$  on the underlying real manifold $X$
such that $\pi^\sharp_R = \pi_I^\sharp \smalcirc J^T$
 \cite{MR2439547}, where $J: TX\to TX$ is the underlying
almost complex structure. Indeed,  
holomorphic Poisson manifolds are equivalent to a special
class of    Poisson--Nijenhuis manifolds, namely those where   the
Nijenhuis tensor is almost complex. Therefore,
holomorphic symplectic local groupoids are  equivalent to
a special   class of symplectic-Nijenhuis local groupoids 
in the sense of Sti\'enon--Xu \cite{MR2276462}. Our goal is to describe an explicit
 construction of such a symplectic-Nijenhuis local groupoid.
For this purpose, it suffices to construct  explicitly
two compatible symplectic structures on the local groupoid.

At this point, we must also mention the recent work 
of Crainic-M{\v{a}}rcu{\c{t}}
\cite{MR2900786}, where they present a very simple  explicit
 construction of a symplectic realization of an arbitrary  $C^\infty$-Poisson
manifold $(M, \pi)$ on an open neighborhood of $T^*M$. In fact, another goal of our paper is to present
a conceptual proof of their  theorem.
The idea is quite simple indeed. Given  a local Lie groupoid $\rgpd$
with Lie algebroid $A$,  it is well known that, by choosing an $A$-connection
on $A$, one can construct a local diffeomorphism --
the exponential map \cite{MR1687747} --- from an open neighborhood of
 the zero section of $A$ onto an open neighborhood of the unit space
in $\rgpd$.
Now, if $\rgpd$ is a  
local symplectic groupoid, its Lie algebroid $A$ is known to be
isomorphic to $\cotalg{M}$ --- see~\cite{MR996653}. 
Pulling back the symplectic form on $\rgpd$, 
which Cattaneo-Felder described explicitly using Poisson sigma models 
\cite[Equation~(3.1), Theorem~3.3 and Section~4.3]{MR1938552},
via such an exponential map, one obtains
a symplectic form on an open neighborhood of the zero section
of the cotangent bundle  $T^*M$. 
One can then verify directly that this symplectic form 
coincides with the one obtained in \cite{MR2900786}
(see also \cite{MR2116732} for some related results).
\color{black}
By applying a combination of techniques developed
in the study of symplectic-Nijenhuis local groupoids
 \cite{MR2276462}  and the theory of Lie bialgebroids and Poisson groupoids
\cite{MR1262213, MR1746902}, we are able to  describe explicitly the
 two compatible symplectic structures on the local groupoid,
and thus obtain the following main result of the paper.

\begin{introthm}\label{thm:symplectic-realization-holomorphic}
Let $\holom{X}$ be a holomorphic Poisson manifold with underlying real smooth manifold $X$, almost
  complex structure $J$, and holomorphic Poisson bivector field
 $\pi\in \Gamma (\wedge^2 T^{1, 0} X)$. 
 Choose an affine connection $\nabla$ on $\smooth{X}$.
 Let $\xi \in  \XX(T^*X)$  be the Poisson
 geodesic vector field of $\nabla$ as 
in Example \ref{ex:brussel}.
 Denote by $\varphi^\xi_t$ the flow of $\xi$ on $T^*\smooth{X}$,
 and $\omegacan$ the canonical symplectic form on $T^*\smooth{X}$.
 The following then holds.
\begin{itemize}
  \item [(i)]
 There is an open neighborhood $\smooth{Y} \subset T^*\smooth{X}$ of the zero section such that  $\uomega_R$ and
 $\uomega_I\in \Omega^2(\smooth{Y})$ given, respectively,  by
\begin{align*}
  \uomega_I & = \int_0^1 (\varphi^\xi_t)^*\omegacan \,dt, \ \ \mbox{and} \\
  \uomega_R & = -\int_0^1 \left(\transpose J\smalcirc 
\varphi^\xi_t\right)^*\omegacan \,dt
\end{align*}
are well-defined  symplectic forms, and the $(1,1)$-tensor
$$\uJ=(\uomega_R\bemol)^{-1} \smalcirc \uomega_I\bemol: TY\to TY$$
is an integrable almost complex structure on $Y$. In particular, $Y$ endowed with $\uJ$ defines a complex manifold $\holom{Y}$.
 \item [(ii)] The 2-form $\uomega\in \Omega^2(Y)\otimes \CC$ defined by
  $$ \uomega := \frac{1}{4}\left(\uomega_R - i \, \uomega_I\right)$$
is holomorphic symplectic on $\holom{Y}$ and the natural projection
$\pr|_Y: \holom{Y} \to \holom{X}$ is a holomorphic symplectic realization
of   $(\holom{X}, \pi)$.
\item  [(iii)] The zero section is a Lagrangian submanifold of
$(\holom{Y}, \uomega)$.
\end{itemize}
Moreover, different choices of  the affine connection
$\nabla$ give rise to isomorphic holomorphic symplectic realizations.
\end{introthm}

Note that if the Poisson structure is trivial (i.e.
$\pi=0$),
 then $( \holom{Y}, \uomega)$ reduces to the canonical holomorphic 
symplectic manifold  $T^*\holom{X}$. Therefore,
 the holomorphic symplectic manifold 
$(\holom{Y}, \uomega)$ can be considered as
 a  deformation of the canonical holomorphic symplectic manifold 
$T^*\holom{X}$ parameterized by the  holomorphic Poisson structure
$\pi$. It would be interesting to investigate how   our result
is related to Kodaira theory of deformation of complex structures \cite{MR2109686}.

The present paper was influenced in large measure by 
Petalidou's splendid work \cite{2015arXiv150107830P} 
on symplectic realizations of non-degenerate Poisson--Nijenhuis manifolds. Making use of the computational approach of \cite{MR2900786}, Petalidou discovered 
an explicit  expression for the 2-forms on 
the symplectic realization.
However, her proof of their compatibility is, to the best of our
 understanding, not entirely sound. 
In our approach, which is more conceptual, 
tracing the hidden underlying groupoid structures
reveals crucial for proving  the compatibility.

Finally, we would like to point out that our approach
draws from various integration results valid only in the 
context of smooth manifolds. It is not clear whether this method
will be of any use in the  context of  algebraic varieties.
So the analogue of Problem~\ref{problem-b} 
for algebraic Poisson varieties remains open.

\section*{Acknowledgements}
We wish to thank  Alberto Cattaneo, Camille Laurent-Gengoux,
Joana Margarida Nunes da Costa, Fani Petalidou, Mathieu Sti\'enon, 
 Izu Vaisman and Alan Weinstein for inspiring discussions and comments.

\section{Holomorphic Poisson Manifolds and Symplectic Realizations}

In this section, we briefly recall, for the sake of completeness, some standard definitions on holomorphic Poisson structures. As most of those elementary notions closely parallel the real smooth Poisson case, we simply point the reader to the appropriate references for further details.

In what follows, let $\holom{X}$ be a complex manifold and $X$ its underlying real manifold. We will denote the structure sheaf of $\holom{X}$ by $\cO_\holom{X}$.
Recall that a  complex structure on $\holom{X}$ is equivalent to
an integrable almost complex structure $J$ on $X$,
i.e. an endomorphism  $J:TX \to TX$
of the underlying  real tangent bundle $TX$
with $J^2=-1$ and  with the vanishing Nijenhuis torsion.
Furthermore, the holomorphic tangent bundle $T\holom{X}$ is
 isomorphic (as a complex vector bundle) to  $T^{1,0}X \subset  TX\otimes \mathbb{C}$.

\begin{definition}\label{dfn:holompoisson}
By   a \emph{holomorphic Poisson structure} on
a complex manifold  $\holom{X}$,
we mean that its structure sheaf 
 $\cO_{\holom{X}}$ is endowed with a bracket
  $$ \{ \cdot, \cdot \}_U : \cO_{\holom{X}}(U) \times \cO_{\holom{X}}(U) 
\to \cO_{\holom{X}}(U), \ \  \  \forall \  U\subset \holom{X} $$
  such that $(\cO_{\holom{X}}, \{\cdot, \cdot\})$ is a sheaf
 of Poisson algebras.
\end{definition}

A \emph{holomorphic Poisson manifold} is a  complex manifold  $\holom{X}$
endowed with a holomorphic Poisson structure.
  As in the smooth case, Definition \ref{dfn:holompoisson} is equivalent to
 a holomorphic Poisson   bivector field on $\holom{X}$.

\begin{proposition}[\cite{MR2439547,MR2545872}]\label{prop:holompoissonbivector}
 Let $\holom{X}$ be a complex manifold with a holomorphic Poisson structure
 $\{\cdot, \cdot\}$. There is a unique bivector field
 $\pi \in \Gamma(\wedge^2 T^{1,0}X)$ satisfying
 \begin{equation}\label{eqn:holompoissonbivector}
   \bar\partial \pi=0 \quad \text{and} \quad [\pi,\pi]=0
 \end{equation}
such that for any open subset $U \subset X$
 and any   $f, g \in \cO_\holom{X}(U)$,
 $$ \{f, g\}_U=\langle \pi, \partial f \wedge\partial  g \rangle.$$

Conversely, any bivector field $\pi \in \Gamma(\wedge^2 T^{1,0} X)$
 satisfying (\ref{eqn:holompoissonbivector})
 defines a unique holomorphic Poisson structure on $\holom{X}$.
\end{proposition}

In particular, $\pi$ is  called a \emph{holomorphic Poisson bivector field}
 on $\holom{X}$ and
 $(\holom{X}, \pi)$ a holomorphic Poisson manifold. 
Note that $\pi$ induces a morphism of holomorphic vector bundles
$\pi^{\#} : T^* \holom{X} \to T \holom{X}$.

The next lemma, which connects holomorphic Poisson structures on $\holom{X}$ with Poisson-Nijenhuis structures on $X$ (see Appendix \ref{app:pn-manifolds}), will be needed in the proof of a slightly more general version of Theorem \ref{thm:symplectic-realization-holomorphic}. 

\begin{lemma}[\cite{MR2439547}]\label{lem:hp-is-pn}
 Let $\holom{X}$ be a complex manifold with almost complex structure $J$.
 Assume that $\pi= \pi_R + i\pi_I \in \Gamma(\wedge^2 T^{1,0}X)$,
where  $\pi_R$ and $\pi_I \in \Gamma(\wedge^2 T X)$.
Then $\pi$ is a holomorphic Poisson tensor if and only if
  \begin{itemize}
  \item[(i)] $(\pi_I, J)$ defines a Poisson--Nijenhuis structure on $X$, and
  \item[(ii)] $\pi^\sharp_R = \pi_I^\sharp \smalcirc \transpose{J}:\ \ T^*X\to
 TX$, where $\transpose{J}: T^*X\to T^*X$ denotes the dual of $J$. 
  \end{itemize}
\end{lemma}
%

  A complex manifold $\holom{X}$ endowed with a holomorphic Poisson bivector field $\pi \in \Gamma(\wedge^2 T\holom{X})$ is called \emph{holomorphic symplectic} if the associated morphism $\pi^\# : T^*\holom{X} \to T\holom{X}$ is invertible. In that case, we also say  that $\pi$ is \emph{non-degenerate}.

\begin{remark}
  If $\pi$ is a non-degenerate holomorphic Poisson bivector field,
for any $k > 0$, 
 $\pi^\#$ extends to an isomorphism
 $$\wedge^k \pi^\# : \wedge^k T^*\holom{X} \to \wedge^k T\holom{X},  
$$
 of holomorphic vector bundles. 
Then  $\omega = (\wedge^2 \pi^{\#})^{-1}(\pi)$ is   a
holomorphic symplectic $2$-form.
\end{remark}

Assume $(\holom{X}, \pi_\holom{X})$ and $(\holom{Y}, \pi_\holom{Y})$ are two holomorphic Poisson manifolds. A
 holomorphic map $f : \holom{X} \to \holom{Y}$ is said to be \emph{Poisson} if the pushforward $f_*(\pi_\holom{X})$ is well defined and
$f_*(\pi_\holom{X}) = \pi_\holom{Y}.$

\begin{definition}
 Let $\holom{X}$ be a holomorphic Poisson manifold.
 A \emph{holomorphic symplectic realization} of $\holom{X}$ is a holomorphic
 symplectic manifold $\holom{Y}$ together with a holomorphic map 
$q : \holom{Y} \to \holom{X}$ such that:
 \begin{itemize}
 \item[1)]  $q : \holom{Y} \to \holom{X}$ is a surjective submersion, and
 \item[2)] $q$ is a Poisson map.
 \end{itemize}
\end{definition}

\begin{example}
 Let $\holom{X}$ be a complex manifold. Let $\pi = 0$ 
be the zero bivector field  on $\holom{X}$.
 Then $\pi$ is a Poisson bivector field and $(\holom{X}, \pi)$ is
 a holomorphic Poisson manifold. The holomorphic
 cotangent bundle $T^*\holom{X}$, endowed with the canonical symplectic structure and the natural projection map $q : T^*\holom{X} \to \holom{X}$ gives a holomorphic symplectic realization of $\holom{X}$.
\end{example}

\begin{example}
Let $\frakg$ be a  finite dimensional complex Lie algebra. Its
complex dual $\frakg^*$ admits a canonical
linear holomorphic Poisson structure, called
{\em Lie-Poisson structure}.
Let  $G$ be a complex Lie  group with Lie algebra $\frakg$.
 Then $G$ is a complex manifold, and
 $T^*G$, equipped with the canonical holomorphic symplectic
structure and the  left translation $q : T^*G\to T^*_{e}G\cong \frakg^*$,
 defines a holomorphic symplectic realization of $\frakg^*$.
\end{example}

\section{Symplectic Local Groupoids: The Cattaneo-Felder Construction}

\renewcommand{\smooth}{\mathcal{C}^\infty}

Let $M$ be a real smooth manifold endowed with a Poisson bivector field 
$\pi\in\Gamma(\wedge^2 TM)$. In this section, we recall an
 explicit construction, 
due to Cattaneo--Felder \cite{MR1938552}, for the symplectic local groupoid 
associated with $(M,\pi)$.
The fundamental idea is to construct it as 
a quotient of the space of all paths of a certain type
in the cotangent Lie algebroid of $M$.
Cattaneo--Felder refer to these paths as 
\emph{solutions of the constraint equation (``Gauss law'')}
\cite[Equation~(3.2)]{MR1938552}\footnote{ 
This construction was subsequently extended to arbitrary Lie algebroids 
in~\cite{MR1973056} and the paths became known as $A$-paths. 
Note that in~\cite{MR1938552}, Cattaneo--Felder did not use the 
terminology `local symplectic groupoids.' 
However, \cite[Theorem 4.7]{MR1938552} essentially gave an explicit 
construction of the \emph{local} symplectic groupoid
since \cite[Assumtion  4.6]{MR1938552} is always satisfied 
in a neighborhood of the unit space $M$.}.
These paths characterize those whose values under the momentum map 
of an infinite dimensional Hamiltonian action vanish.
\color{black}
 The construction can be conceptually separated in 
two parts. The first is valid for an arbitrary Lie algebroid and constructs a local Lie groupoid out of a  Lie algebroid (Theorem \ref{thm:local-groupoid}). The second explicitly deals with the symplectic structure by inducing a symplectic form on the local groupoid constructed in the first part (Theorem \ref{thm:cattaneo-felder}).

Let $I = [0, 1]$ be the closed unit interval, and $n$ the dimension of $M$.
 For a smooth vector bundle $E \to M$ of rank $k$,
  consider the space $\tP^p(E) = \mathcal{C}^p(I, E)$ of $C^p$-paths
 valued in $E$. It can be endowed with the structure of a
smooth Banach manifold \cite{MR772023} by choosing a trivializing $\mathcal{C}^\infty$-atlas
 $(\varphi_i : E|_{U_i} \to \RR^n \times \RR^k)_{i \in J}$ for $E$ and defining a family $(\tilde{\varphi_{i}})_{i \in J}$ by
\begin{align*}
  \tilde{\varphi_{i}} : & \; \mathcal{C}^p (I, E|_{U_i}) \to
\mathcal{C}^p (I, \RR^n \times \RR^k) :  f \mapsto \varphi_{i} \smalcirc f.
\end{align*}
It is easily checked that the change of charts $\tilde\varphi_i \circ \tilde\varphi_j^{-1}$ are indefinitely  Frechet-differentiable with respect to the 
$\mathcal{C}^p$-norm on $\mathcal{C}^p(I, \mathbb{R}^{n+e})$,
 and therefore the family $\{\tilde\varphi_i\}$ induces an atlas for paths
 that fit in a single trivializing local chart for $E$. It is straightforward to extend it to
 an atlas for \emph{all} paths and thus $\tP^p(E)$ is an infinite dimensional
 smooth (i.e. a $C^{\infty}$-)
 Banach manifold.
\color{black}

Let $A$ be a Lie algebroid over $M$ with projection $p : A \to M$ and anchor $\rho: A\to TM$. 
In what follows,
 we will mostly be concerned with the space $\tP^1(A)$ of
 $\mathcal{C}^1$-paths valued in $A$. We will abbreviate the notation by letting $\tP(A) = \tP^1(A)$. Recall that an element $a :I \to A$ 
in  $\tP(A)$ is called an {\em $A$-path}
 if \begin{equation}\label{eq:Apath}
\rho \big( a(t) \big) = \frac{\diff \gamma (t)}{\diff t},
\end{equation}
where $\gamma (t)=(p \smalcirc a)(t)$ is the base path. We will denote by
$\P(A)$ the set of all $A$-paths. It is easy to see that $P(A)$ is a closed infinite dimensional Banach submanifold of $\tP (A)$. 

In a way that closely parallels the case of finite dimensional manifolds, one can define  \cite{MR772023} the tangent bundle $T\tP(A)$ of $\tP(A)$ as a certain collection of derivations. However, for what we will need, it is enough to
 recall that 
 there exists a natural isomorphism 
$$\tau : T\tP(A) \to \tP(TA)$$
 of the tangent bundle of $\tP(A)$ with $\mathcal{C}^1$-paths valued in $TA$. Explicitly, for a given $v\in T\tP(A)$, choose a path $\theta : I \to \tP(A)$ such that $v = \left.\frac{d}{ds}\right|_{s=0}\theta_s$. Then

\begin{equation}\label{eq:iota-map}
  (\tau v)(t) \equiv \left.\frac{d}{ds}\right|_{s=0} 
\left(\theta_s(t)\right) \in T_{\theta_0(t)}A.
\end{equation}
Fibrewise, $\tau$ then gives an isomorphism
$$\tau : T_a \tP(A) \to \{ X \in \tP(TA) \mid X(t) \in T_{a(t)}A \}$$
for all $a : I \to A$ in $\tP(A)$.


Now let $\rgpd\toto M$ be a local Lie groupoid with Lie algebroid $A$, source 
and target maps $\alpha,\beta : \rgpd \to M$, and unit map $\epsilon : M \to \rgpd$.
 Let $\exp : \Gamma(A) \to \text{Bis}(\rgpd\toto M )$ be the usual 
exponential map, where $\text{Bis}(\rgpd\toto M)$ is
 the set of local bisections 
of $\rgpd\toto M$ \cite{MR2157566}. Recall that 
$\text{Bis}(\rgpd\toto M)$ acts on $A$ by the differential
 of the  conjugation map. Let us denote this action by  
$$\Ad : \text{Bis}(\rgpd\toto M) \to \Aut(A).$$
Set
$$ \left.\had(X)\right|_{a_0} := \left.\frac{d}{dt}\right|_{t=0} \Ad_{\exp(tX)} (a_0) $$
for all $X\in \Gamma(A)$ and all $a_0\in A$. In particular,
 we have a map $\had : \Gamma(A) \to \XX(A)$,
 where $\XX(A)$ denotes the space of all vector fields on $A$.

It is well known \cite{MR1938552, MR1973056, MR2063018,  MR2911881} that
 $\rgpd\toto M$ can be reconstructed as a quotient of $P(A)$ by a certain integrable distribution $\cD(PA) \subset T \P(A)$.
More explicitly,  for  any $a\in P(A)$, denote
  $$ H_{a} := \left\{ \left[ t \mapsto \had(\xi(t))\vert_{a(t)} +
 \left.\frac{d\xi(t)}{dt}\right\vert_{\gamma(t)} \right] \in \tP(TA) \;\middle|\;
\forall  \xi : I\to \Gamma(A), \; \xi(0)=\xi(1)=0 \right\}, $$
where $\gamma(t)$ is the base path of $a(t)$, and 
$$\left.\frac{d\xi(t)}{dt}\right|_{\gamma(t)} \in A_{\gamma(t)}$$ 
is naturally identified with a  vertical tangent
 vector in $T_{a(t)}A$. Define,  $\forall a\in \P(A)$,
$$\cD_a(PA) := \tau^{-1}H_a,$$
where $\tau : T\tP(A) \to \tP(TA)$ is the  isomorphism 
  \eqref{eq:iota-map}. The most important
 facts we will need are summarized  in the following theorem. For
 details see  \cite{MR1973056}.

\begin{theorem}
\label{thm:local-groupoid}
The following statements hold.
\begin{itemize}
\item[(i).] $\cD(PA)$ is a finite codimensional integrable distribution on $P(A)$.
\item[(ii).] Let $\cF(A)$ be the foliation integrating $\cD(PA)$. Then there is an open neighborhood $\Pl(A)\subset \P(A)$ of the natural embedding of $M$ into $P(A)$ as constant paths, where the space of leaves 
  $$\bPl(A) := \Pl(A)/(\cF(A)\cap \Pl(A))$$ 
  is a finite dimensional smooth manifold.  By
\begin{equation}
\label{eq:q}
q: \Pl(A)\to \bPl(A)
\end{equation}
we denote the quotient map.

\item[(iii).]  The maps
  \begin{align*} & \alpha : \P(A) \rightarrow M : a \mapsto a(0), \\
                 &  \beta : \P(A) \rightarrow M : a \mapsto a(1), \\
                 &  \epsilon : M \rightarrow \P(A) : m \mapsto a(t) \equiv \mathbf{0}_{m},
  \end{align*}
    descend to smooth maps $\bar\alpha : \bPl(A) \to M, \bar\beta : \bPl(A) \to M$, and $\bar \epsilon : M \to \bPl(A)$ on the quotient.
 Furthermore, there is an open neighborhood of the constant 
diagonal embedding
 $$ M \hookrightarrow M \times M \hookrightarrow \Pl(A)
 \times_{\beta,M, \alpha} \Pl(A) $$
    where the concatenation operation on paths induces a well defined local multiplication
    $$\bar \mu :  \bPl(A) \times_{\bar\beta,M, \bar\alpha} \bPl(A) \to \bPl(A)$$
on $\bPl(A)$. Finally,
 with $\bar \mu$ as multiplication, and $\bar\alpha, \bar\beta$ and 
$\bar\epsilon$ as, respectively, source, target and unit maps,
 $\bPl(A) \toto M$ has the structure of a local Lie groupoid
 with Lie algebroid $A$.
    \end{itemize}
\end{theorem}

When $A$ is the cotangent Lie algebroid $(T^*M)_\pi$ of
a smooth Poisson manifold $(M, \pi)$, one obtains the following theorem.

\begin{theorem}[\cite{MR1938552, MR2128714}] \label{thm:cattaneo-felder} 
  Let $(M, \pi)$ be a Poisson manifold,
 and let $A$ denote its corresponding  cotangent
 Lie algebroid $\cotalg{M}$.  The following statements hold.
  \begin{itemize}
  \item[(i).] For all $a\in \tP(A)$,
 and all  $u,v \in T_a \tP(A)$, define
    \begin{equation}\label{eq:cattaneo-felder-form}
      \tomega_\can( u, v ) = \int_0^1 
      \omega_\can ((\tau u)(t), (\tau v)(t) ) dt.
    \end{equation}
Then $\tomega_\can$    is a  symplectic form on $\tP(A)$.
Here $\omega_\can \in \Omega^2 (T^*M)$ 
denotes the canonical symplectic form on $T^*M$.
  \item[(ii).] There exists a symplectic form $\bomega$
    on $\bPl(A)$, with which  the local groupoid $\bPl(A) \toto M$
 from Theorem \ref{thm:local-groupoid} (iii) 
becomes a symplectic local groupoid. 
Moreover, we have 
    \begin{equation}\label{eq:cattaneo-felder-reduction}
      q^* \bomega = \iota^* \tomega_\can,
    \end{equation}
    where $q: \Pl(A) \to \bPl(A)$ is  the quotient map, and
    $\iota: \Pl(A)\hookrightarrow \tP(A)$ is the natural inclusion.
  \end{itemize}
\end{theorem}


\begin{remark}
Some historical remarks are in order.
Part~(i) of Theorem~\ref{thm:cattaneo-felder}
was proved in~\cite{MR1938552} --- see~\cite[Equation~(3.1), Theorem~3.3 and
Section~4.3]{MR1938552}. Part~(ii) was explicitly proved in~\cite{MR1938552} 
in the case of an integrable Poisson manifold --- see~\cite[Theorem 4.7]{MR1938552}.
However, by restricting to a neighborhood of the unit space, 
one can adapt the argument to prove the existence of a local symplectic groupoid 
integrating a given Poisson manifold, 
since \cite[Assumption 4.6]{MR1938552} holds automatically.
This was done in full details in~\cite{MR2128714}.
\end{remark}
\color{black}

Before we close this section, let us record the following proposition, which
 will  be needed later on. Its proof is straightforward and  follows
 immediately from the standard construction of  $\bPl(A)$.

\begin{proposition}\label{prop:paths-to-groupoid}
  Let $A$ and $B$ be  Lie algebroids over the same base manifold $M$,
 and let $\psi : A \to B$ be a Lie algebroid morphism over the identity map. 
\begin{itemize}
  \item[(i)] The induced map on path spaces
$$ \tP(\psi) : \tP(A) \to \tP(B) : [t\mapsto a(t)] \mapsto [t \mapsto \psi(a(t))]  $$
preserves $A$-paths,  and descends to a morphism of
local Lie groupoids $$\bar P(\psi) : \bPl(A)  \to \bPl(B)$$
 making the diagram
\begin{equation}\label{eq:paths-to-groupoid-diagram}
  \begin{gathered}
  \begin{tikzpicture}
    \matrix (m) [matrix of math nodes, row sep=3em, column sep=3em] {
      \Pl(A) & \Pl(B) \\
      \bPl(A) & \bPl(B) \\
    };
    \path[-stealth]
    (m-1-1) edge node [above] {$\tP(\psi)$} (m-1-2)
    (m-1-2) edge node [right] {$q'$} (m-2-2)
    (m-1-1) edge node [left] {$q$} (m-2-1)
    (m-2-1) edge node [below] {$\bar P(\psi)$} (m-2-2);
  \end{tikzpicture}
\end{gathered}
\end{equation}
commute.
Here $q$ and $q'$ are the respective quotient maps as in \eqref{eq:q}.
\item [(ii)] The diagram
  \begin{equation}\label{eq:paths-to-groupoid-tangent-diagram}
    \begin{gathered}
      \begin{tikzpicture}
        \matrix (m) [matrix of math nodes, row sep=3em, column sep=3em] {
          T \P(A) & T \P(B) \\
          \tP(TA) & \tP(TB) \\
        };
        \path[-stealth]
        (m-1-1) edge node [above] {$\tP(\psi)_*$} (m-1-2)
        (m-1-2) edge node [right] {$\tau \circ \iota_* $} (m-2-2)
        (m-1-1) edge node [left] {$\tau \circ \iota_*$} (m-2-1)
        (m-2-1) edge node [below] {$ \tP(\psi_*)$} (m-2-2);
      \end{tikzpicture}
    \end{gathered}
  \end{equation}
  commutes. Here, by abuse of notations,
$\iota$ denotes both  embeddings
$\P(A)\to \tP( A)$ and $\P(B)\to \tP( B)$, and  $\tP(\psi_*) :
 \tP(TA) \to \tP(TB)$ is the map induced, as in part (i), from the tangent map $\psi_* : TA \to TB$.
\end{itemize}
\end{proposition}

\section{Exponential Maps}

In Lie theory, the classical exponential map establishes 
a local diffeomorphism  from an
  open neighborhood of zero in a Lie algebra
to the corresponding local Lie group.
This construction extends to Lie algebroids and local Lie groupoids. 
Unlike the Lie algebra case, however, one needs to choose 
some geometrical structure, namely an $A$-connection on $A$. 
In this section, we recall some basic facts about the exponential map for
 Lie groupoids, and  describe the latter  explicitly   in
 the case of the local Lie groupoid of Theorem \ref{thm:local-groupoid} (iii).

Let $A$ be, as before, a Lie algebroid over $M$. By an $A$-connection on $A$ we mean an $\mathbb{R}$-bilinear map 
$$\nabla : \Gamma(A) \times \Gamma(A) \to \Gamma(A): (X,Y) \mapsto \nabla_X Y$$ 
satisfying the  conditions
\begin{align*}
  &\nabla_{f X} Y = f \nabla_X Y, \quad    \text{and} \\
  &\nabla_X (fY)  = (\rho(X)f) Y + f \nabla_X Y,
\end{align*}
for all $X,Y\in \Gamma(A)$ and $f\in \mathcal{C}^\infty(M)$.

\begin{example}
Any linear connection   $\tilde\nabla$ on the 
vector bundle  $A$  induces
 an associated $A$-connection on $A$ by the formula $\nabla_X Y =\tilde\nabla_{\rho (X)} Y$.
However, not every $A$-connection on $A$ is of this form.
\end{example}

\begin{definition}
  An \emph{$A$-geodesic} (or a \emph{geodesic $A$-path}) is an $A$-path $a : I \to A$ satisfying the geodesic equation:
  $$\nabla_{a(t)} a(t)=0$$
  for any $t \in I$.
\end{definition}

An $A$-connection on $A$ also defines a map $h : A\times_M A \to TA$, called a \emph{horizontal lifting} \cite{2014arXiv1408.2903L}:
$$ h(a,b) = \bar b_* ( \rho(a) ) - \tau_b\big ( (\nabla_a \bar b)|_x\big)
 \in T_b A, $$
for any  $x\in M$ and $a, b\in A_x$. Here
$ \bar b \in \Gamma(A)$ is any section   satisfying
 $\bar b( x ) = b$, and $\tau_b$
denotes the canonical linear isomorphism between the fiber $A_x$
and the vertical tangent space of $A$ at the point $b$.
 It is not hard to check that
 $h(a,b)$ does not depend on the choice of the extension  $\bar b$.

\begin{definition}\label{def:geodesic-vector-field}
  The \emph{geodesic vector field of $\nabla$} is the vector field $\xi \in \XX(A)$ defined by
  $$\xi_a = h(a,a)$$
  for any $a \in A$.
\end{definition}

In what follows, for a  given  $A$-connection $\nabla$ on $A$, we will denote by $\varphi^\nabla_t$ the flow of its geodesic vector field.

\begin{proposition}\label{prop:algebroid-exponential}
    Let $A$ be a Lie algebroid, and $\nabla$ an $A$-connection on $A$. The following holds.
  \begin{itemize}
  \item [(i).] There is a neighborhood $U\subset A$ of the zero section such that $\varphi_t^\nabla$ is defined for all $t\in I$ and,
  \item [(ii).] for all $a_0 \in U$, the path $[t\in I \mapsto a(t) = \varphi_t^\axi (a_0)]$
    is $A$-geodesic.
  \end{itemize}
\end{proposition}

\begin{proof}
  \begin{itemize} 
  \item [(i)] Denote by $m_s : A \to A$ the fibrewise
scalar multiplication by $s\in \RR$. Let $\xi\in \XX(A)$ be the geodesic vector field of $\nabla$.
It is easily checked that $ s\xi_a = (m_s)^{-1}_* \xi_{sa}$ for 
all $s>0$ and all $a\in A$.
It then follows that $$ s\varphi_{ts}^\nabla(a) = \varphi_t^\nabla(sa), $$
    where one side is defined exactly when the other is. Rescaling
 locally,  this yields the claim.
  \item[(ii)] Fix  any $a_0 \in U \subset  A $,
 and let $a(t) = \varphi^\axi_t(a_0)$. Denote by $\gamma(t) = p(a(t))$ the underlying base path. We have
    \begin{align*}
      p_*(\dot a(t)) = p_*(\xi(a(t))) = p_*(h(a(t), a(t))) = \rho(a(t)).
    \end{align*}
    Hence $a(t)$ is  indeed  an $A$-path. Choose any time-dependent section $\bar a : I \times M \to A$ such that $\bar a(t, p(a(t))) = a(t)$. Then 
    \begin{align*}
      \nabla_{a(t)} a(t) & = \frac{\partial}{\partial t} \bar a(t, \gamma(t)) + \nabla_{a(t)} \bar a (t, \gamma(t)), \\
                         & = \left[ \dot a (t) - \bar a_*^t(\dot \gamma(t))\right] + \nabla_{a(t)} \bar a (t, \gamma(t)) \\
                         & = \dot a (t) - \xi(a(t)) \\
                         & = 0.
    \end{align*}
    Thus the conclusion follows.\qedhere
\end{itemize}
\end{proof}

\begin{example}
\label{ex:brussel}
  Let $(M,\pi)$ be a Poisson manifold, 
and $(T^*M)_\pi$ its cotangent Lie algebroid.
Choose  an affine connection   $\nabla^{TM} : \XX(M)\times\XX(M)\to \XX(M)$ 
 on $M$. Let $\nabla^{T^*M}: \XX(M) \times \Gamma(T^*M)\to \Gamma(T^*M)$ be
 the corresponding linear connection on $T^*M$--the dual bundle of
$TM$.
Introduce a Lie algebroid $(T^*M)_\pi$-connection
$\nabla : \Gamma(T^*M) \times \Gamma(T^*M) \to \Gamma(T^*M)$
on $(T^*M)_\pi$ by
 $$\nabla_\lambda \nu = \nabla^{T^*M}_{\pi^\sharp(\lambda)} \  \nu ,
 \quad\quad \forall \lambda,\nu\in \Gamma(T^*M).$$
  In local coordinates $\{q^i\}$ on $M$,  assume that
  $$ \nabla_{\frac{\partial}{\partial q^i}}^{TM} \,
\frac{\partial}{\partial q^j} = \sum_k \Gamma^k_{ij}\frac{\partial}{\partial q^k} \;\text{ and }\; \pi =
\sum_{ij} \pi^{ij} \frac{\partial}{\partial q^i}\wedge \frac{\partial}{\partial q^j}. $$
  Then the corresponding
 geodesic vector field $\xi\in\XX(T^*M)$  has the local expression:
  $$ \xi = \sum_{ij} p_i \pi^{ij} \frac{\partial}{\partial q^j} +
\sum_{ijkl} p_k p_l \pi^{ki}\Gamma_{ij}^l\frac{\partial}{\partial p_j}, $$
  where $\{q^i,p_i\}$ are the induced local coordinates on $T^*M$.
We call $\xi$ the {\em Poisson geodesic vector field} of $\nabla^{TM}$
(it was called Poisson spray in \cite{MR2900786}).
\end{example}

Let $\nabla$ be an $A$-connection on $A$,
 and  $\rgpd \toto M$  
a  local Lie groupoid integrating $A$ with source and target maps
 $\alpha$ and $\beta$, respectively. For any $x\in M$,
 there is an affine connection on the source fiber $\rgpd_x=\alpha^{-1} (x)$, which we will denote by $\bar \nabla^x$. It is defined 
\cite{MR1687747} uniquely by
\begin{equation}\label{eq:nablabar}
  \left.\bar\nabla^x\right._{X^L|_{\alpha^{-1} (x)}} ( Y^L|_{\alpha^{-1} (x)}) = (\left.\nabla \right._X Y)^L|_{\alpha^{-1} (x)},
\end{equation}
%
for any  $X,Y\in \Gamma(A)$. Here $X^L$ denotes the left-invariant
 vector field on $\rgpd$ associated to $X$.

\begin{definition}[\cite{MR1687747}]\label{def:groupoid-exponential}
  Let $\nabla$ be an $A$-connection on $A$,
and $\rgpd\toto M$ a local Lie groupoid with Lie algebroid $A$.
 The \emph{groupoid exponential map}  defined by $\nabla$ is the map
 $\exp^\nabla : A \to \rgpd$, defined on a neighborhood of the zero section
 in $A$, and which, on each fiber $A_x$,
 is given by the ordinary exponential map of
 the affine connection $\bar\nabla^{x}$ on $\rgpd_x$. 
\end{definition}

When no risk of ambiguity exists, we shall simply denote
 ``$\exp^\nabla$'' by ``$\exp$'', hiding the dependency on the 
$A$-connection $\nabla$.

It can be proved that $\exp$ is smooth \cite{MR1687747}. Also
 note that, by definition, $\alpha \circ \exp = p$,
 where $p: A \to M$ is the base point projection. In particular, for any $a_0\in A$, the exponential path 
$t\mapsto \exp(ta_0)$
is a source-path in $\rgpd$.

Letting $U \subset A$ as in Proposition \ref{prop:algebroid-exponential} (i), we define
\begin{equation}\label{eq:phi-map}
  \Phi : U \to \P(A) : a_0\in U \mapsto [t\in I \mapsto \varphi^\axi_t(a_0)],
\end{equation}
i.e.\ $\Phi(a_0)$ is the $A$-geodesic stemming from $a_0$. One should think of $\Phi$ as a kind of exponential map
 at the level of A-paths \cite{MR1973056}.
 Formally, the relation between $\Phi$ and the groupoid exponential
map  of Definition \ref{def:groupoid-exponential} is summarized
 in Proposition \ref{thm:exp-local-diffeo}. Its proof is a consequence of the
 following simple lemma,
 which relates, for a given Lie algebroid element $a_0\in A$, the groupoid exponential path $\exp(ta_0)$ to the $A$-geodesic $\varphi^\nabla_t(a_0)$ stemming from $a_0$.

\begin{lemma}\label{prop:groupoid-exponential}
  Let $U\subset A$ be as in Proposition \ref{prop:algebroid-exponential} (i) and 
fix any $a_0 \in U$. Let
 $a = \Phi(a_0)$ be the associated geodesic $A$-path, i.e.\ $a(t) = \varphi^\axi_t(a_0)$. Also let 
$ r(t) = \exp(t a_0)$ 
be the associated exponential path in $ \rgpd$. Then $r(t)$
is a source-path that satisfies the conditions:
 $$ 
 \begin{cases}
 [L_{r^{-1}(t)}]_*\dot r(t)  = a(t) \quad \forall t\in I, \text{and} \\
 r(0) = \epsilon(p(a_0)), \; \dot r(0) = a_0. 
\end{cases}
$$
 Here, for any
 $g\in \rgpd$, the map $L_g : \rgpd_{\beta(g)} \to \rgpd_{\alpha(g)}$
 denotes the  left multiplication by $g$.
\end{lemma}
\begin{proof}
  From \eq{eq:nablabar}, it follows that
  \begin{align*}
    0 & = \left.\bar\nabla^x\right._{\dot r(t)} \dot r(t) = [L_{r(t)}]_*\left( \nabla_{[L_{r^{-1}(t)}]_* \dot r(t)}\;[L_{r^{-1}(t)}]_* \dot r(t) \right).
  \end{align*}
  Hence the $A$-path $[L_{r^{-1}(t)}]_*\dot r(t)$ is $A$-geodesic. Since we also have $$[L_{r^{-1}(0)}]_* \dot r(0)=\epsilon({p(a_0)}) \cdot r(0) = a_0,$$ the result follows from the unicity of geodesics. 
\end{proof}


The proof of the following proposition
 is a straightforward consequence of the construction of the local groupoid $\bPl(A)\toto M$ of Theorem \ref{thm:local-groupoid} 
combined with
 Proposition \ref{prop:algebroid-exponential} 
and Lemma \ref{prop:groupoid-exponential}
 (see \cite{MR1973056, MR2911881}).

\begin{proposition}\label{thm:exp-local-diffeo}
  Let $\nabla$ be an $A$-connection on $A$ and $U\subset A$ be as in Proposition \ref{prop:algebroid-exponential}. Then, up to choosing a sufficiently small 
open subset $\Pl(A) \subset \P(A)$ as in
Theorem \ref{thm:local-groupoid} (ii), the restriction of the groupoid 
exponential map
 $\exp|_{U} : U \to \bPl(A)$ is a diffeomorphism onto its image. 
Moreover, the diagram 
    \begin{equation}\label{eq:leaf-space-exp}
      \begin{gathered}
      \begin{tikzpicture}
        \matrix (m) [matrix of math nodes, row sep=3em, column sep=3em] {
          U & \Pl(A)\\
          \vspace{0.1em} & \bPl(A) \\
        };
        \path[-stealth]
        (m-1-1) edge node [above] {$\Phi$} (m-1-2)
        (m-1-2) edge node [right] {$q$} (m-2-2)
        (m-1-1) edge node [below] {$\exp\quad$} (m-2-2);
      \end{tikzpicture}
    \end{gathered}
    \end{equation}
commutes. Here $q : \Pl(A) \to \bPl(A)$ is the quotient map as in  
\eqref{eq:q}.
\end{proposition}

The following simple technical lemma will be useful
 in our subsequent discussions.

\begin{lemma}
  Let $\nabla$ be an $A$-connection on $A$,
 $U\subset A$ and $\Pl(A)\subset \P(A)$ as in   Proposition
\ref{thm:exp-local-diffeo}. Let $\psi : A \to A$ be a morphism of 
Lie algebroids over the identity map,
 and $\Phi : U \to \Pl(A)$ be the map as in \eqref{eq:phi-map}. 
  \begin{itemize}
  \item[(i)] For any $a\in U$ and $v\in T_aA$,  we have
    \begin{equation}
      \label{eq:brussels}
      \mathrm{ev}_t(\tau (\iota_* (\Phi_*(v))))= (\varphi^\axi_t)_*(v), \ \ 
\forall t\in I.
    \end{equation}
  Here $\mathrm{ev}_t$ denotes  the evaluation map
 of a path at time $t$. Also recall that $\iota$ denotes the embedding $\P(A) \hookrightarrow \tP(A)$.
    %
  \item [(ii)] With the notation of Proposition \ref{prop:paths-to-groupoid}, we have, for any $a \in U$ and any $v\in T_a A$:
    \begin{equation}
      \label{eq:zaventem}
      \mathrm{ev}_t (\tau (\iota_*(\tP(\psi)_*(\Phi_*(v)))))= (\psi_*\circ (\varphi^\axi_t)_*)(v), \ \ \forall t\in I.
    \end{equation}
  \end{itemize}
\end{lemma}
\begin{proof} \leavevmode
  \begin{enumerate}
  \item [(i)] This follows  immediately  from \eq{eq:iota-map}.
  \item [(ii)]  By commutative diagram
 (\ref{eq:paths-to-groupoid-tangent-diagram}), we have $\tau \smalcirc \iota_*\smalcirc \tP(\psi)_* =\tP(\psi_*)\smalcirc \tau \smalcirc \iota_*$.
   Hence 
    $$\text{ev}_t(\tau( \iota_*  (\tP(\psi)_*  (\Phi_*(v)))))
    =\text{ev}_t(\tP(\psi_*) (\tau( \iota_*(\Phi_*(v))))))
    =(\psi_*\circ (\varphi_{t}^\axi)_*)(v),$$
    as claimed.\qedhere
\end{enumerate}
\end{proof}

\section{Symplectic Realizations of Poisson Manifolds}

Let $(M,\pi)$ be a Poisson manifold,
 and $A$ its cotangent Lie algebroid $(T^*M)_\pi$. Consider the symplectic
 local groupoid $(\bPl(A)  \toto M, \bomega)$ as in Theorem
 \ref{thm:cattaneo-felder} (ii).

Now, fix $\nabla$ an $A$-connection on $A$,
 and let $U\subset A$ be a sufficiently small open
 neighborhood of the zero section as in Proposition 
\ref{thm:exp-local-diffeo}. Set 
\begin{equation}\label{eq:pullback-omegabar}
  \pomega := \exp^* \bomega
\end{equation}
to be the pullback of $\bomega$ by the groupoid exponential map.
 Then  $\pomega$ is a symplectic form on $U$.

\begin{proposition}
\label{prop:paris}
The symplectic form $\pomega$  can be explicitly expressed as follows: 
\begin{equation}\label{eq:pullback-omegabar-formula}
  \pomega = \int_0^1 (\varphi^\axi_t)^*\omega_\can \; dt,
\end{equation}
where $\omega_\can \in \Omega^2 (T^*M)$
is the canonical symplectic form on $T^*M$, and
 $\varphi^\axi_t$ is the flow of the geodesic vector field $\xi \in \XX (A)$
corresponding to $\nabla$. 
\end{proposition}
\begin{proof}
According to the commuting diagram \eqref{eq:leaf-space-exp},
 we have $\exp^* = \Phi^* \circ q^*$, and by
 \eq{eq:cattaneo-felder-reduction}, we have
 $q^*\bar\omega = \iota^*\tilde{\omega}_{\text{can}}$. Thus
  \begin{equation*}
    \uomega = \Phi^* \iota^*\,\tilde{\omega}_{\text{can}}. 
  \end{equation*}
  On the other hand, $\forall a\in U$ and  $\forall u,v\in T_aA$,
we have
  \begin{align*}
   \big( \Phi^* \iota^*\tilde{\omega}_{\text{can}}\big)(u,v) & = \int_0^1 \omega_{\text{can}}([\tau (\iota_*(\Phi_*(u)))](t), [\tau (\iota_*(\Phi_*(v)))](t))\, dt \\
    & = \int_0^1 \omega_{\text{can}}((\varphi^\nabla_t)_*(u), (\varphi^\nabla_t)_*(v)) \, dt \\
    & = \int_0^1 \big( (\varphi^\nabla_t)^*\omega_{\text{can}} \big) (u,v) \, dt,
  \end{align*}
where we used  \eq{eq:cattaneo-felder-form} for the first equality, and
 \eq{eq:brussels} for the second equality.
The    conclusion thus follows.
\end{proof}

As an immediate consequence of \eq{eq:pullback-omegabar-formula},
 we recover the following theorem,  part (i) of which
was proved  by Crainic-M{\v{a}}rcu{\c{t}}
 by a direct computation \cite{MR2900786}.
See also \cite{MR2116732} for related results.

\begin{theorem}\label{thm:symplectic-realization-nonPN}
Let $(M,\pi)$ be a Poisson manifold
and $A=(T^*M)_\pi$ its cotangent Lie algebroid.
Fix $\nabla$ an $A$-connection on $A$ and let $\xi \in \XX(A)$ be the associated geodesic vector field.
Also let $U\subset A$ be, as in 
Proposition \ref{thm:exp-local-diffeo}, a sufficiently small 
 open neighborhood of the zero section in $A$ so that, in particular, the flow $\varphi^\nabla_t(a_0)$ of $\xi$ is defined for all $t\in I$,
and all  $a_0 \in U$. Then, 
\begin{itemize}
\item[(i)] the projection $\pr|_U: U \subset T^*M \to M$ together
with  the symplectic form $\uomega\in \Omega^2(U)$,
  as defined by \eq{eq:pullback-omegabar-formula}, is a symplectic realization
 of $(M,\pi)$; and
\item [(ii)] the zero section of $T^*M$ is a Lagrangian submanifold of $U$. 
\end{itemize}
\end{theorem}

The geodesic vector field $\xi \in \XX(A)$ is 
 called a \emph{Poisson spray} in \cite{MR2900786, 2015arXiv150107830P}.

\section{Symplectic-Nijenhuis Local Groupoids}\label{sec:symplectic-nijenhuis-local-groupoids}

There is a one-to-one correspondence between Poisson manifolds and 
 symplectic local groupoids. This
 is in fact a special case of the Mackenzie-Xu correspondence
 (Theorem \ref{thm:mackenzie-xu}) recalled in the appendix below.
Such a  correspondence can also be extended to a one-to-one correspondence
 between Poisson-\emph{Nijenhuis} manifolds and 
 symplectic-\emph{Nijenhuis} local groupoids. This   result is due to
Sti\'enon--Xu \cite{MR2276462},
 which we recall in Theorem \ref{thm:symplectic-nijenhuis-groupoids}. In this section, we briefly go over the main idea of  its
 proof.

Let $\rgpd \toto M$ be a local Lie groupoid with source and target maps $\alpha : \rgpd \to M$ and $\beta : \rgpd \to M$, respectively, and with unit map $\epsilon : M \hookrightarrow \rgpd$. Recall that a $(1,1)$-tensor
 $\bN: T\rgpd \to T\rgpd$ on $\rgpd$ is said to be \emph{multiplicative}
 if it defines  a morphism of local Lie groupoids

\begin{equation}
\label{eq:Paris}
  \begin{gathered}
  \begin{tikzpicture}
    \matrix (m) [matrix of math nodes, row sep=3em, column sep=4em] {
      T\rgpd & T\rgpd \\
      TM & TM \\
    };
    \path[-stealth]
    (m-1-1) edge node [above] {$\bN$} (m-1-2)
    (m-1-2) edge [transform canvas={xshift=-0.3em}] node [left] {$\alpha_*$} (m-2-2)
    (m-1-2) edge [transform canvas={xshift=0.3em}] node [right] {$\beta_*$} (m-2-2)
    (m-1-1) edge [transform canvas={xshift=-0.3em}] node [left] {$\alpha_*$} (m-2-1)
    (m-1-1) edge [transform canvas={xshift=0.3em}] node [right] {$\beta_*$} (m-2-1)
    (m-2-1) edge node [below] {$\bN|_{\epsilon_*(TM)}$} (m-2-2);
  \end{tikzpicture}
\end{gathered}
\end{equation}

Here $T\rgpd\toto TM$ is the tangent local groupoid. Note that it is implicitly assumed, as part of the condition, that  $\bN(\epsilon_*(TM)) \subset \epsilon_*(TM)$.

\begin{definition}\label{def:symplectic-nijenhuis-groupoids}
A \emph{symplectic-Nijenhuis local groupoid} is a symplectic local  groupoid
$(\rgpd \toto M, \bomega)$ equipped with a  multiplicative
$(1, 1)$-tensor $\bN: T\rgpd \to T\rgpd$ such that the triple
$(\rgpd, \bomega, \bN )$  is a  symplectic-Nijenhuis manifold.
\end{definition}

\begin{remark}\label{rem:symplectic-nijenhuis-local-groupoids-two-poisson}
  Any symplectic-Nijenhuis local groupoid defines two Poisson local groupoid
 structures on the same underlying local groupoid $\rgpd \toto M$. 
Indeed,
 let $(\rgpd \toto M, \bomega, \bN)$ be a symplectic-Nijenhuis local groupoid and denote by $\bpi \in \Gamma(\wedge^2 T\rgpd)$ the Poisson bivector field given by  inverting  $\bomega$. 
Then the pair $(\rgpd\toto M, \bpi)$ is a Poisson local groupoid.
 Moreover,
 from Proposition \ref{prop:second-poisson-pn}, it follows that
 the bivector field $\bpi_{\bN}$ defined by 
\begin{equation}
\label{eq:6.2}
\bpi_{\bN}^\sharp = \bN \smalcirc \bpi^\sharp
\end{equation}
 is another multiplicative Poisson structure on $\rgpd$,
 and thus in particular gives another Poisson local groupoid $(\rgpd \toto M, \bpi_{\bN})$. Note that, in general, 
 the Nijenhuis tensor $\bN: T\Sigma \to  T\Sigma$ may  not  be invertible,
 and therefore the Poisson bivector field $\bpi_{\bN}$ is not
necessarily  non-degenerate.
 In particular, we do not 
automatically have two symplectic groupoid structures on $\rgpd \toto M$.
\end{remark}

The following theorem is due to Sti\'enon--Xu \cite{MR2276462}.

\begin{theorem}\mbox{}
\label{thm:symplectic-nijenhuis-groupoids}
\begin{itemize}
  \item [(i)] 
 The unit space of a symplectic-Nijenhuis local
 groupoid inherits an induced  Poisson--Nijenhuis
 manifold structure.
\item [(ii)]
 Given a Poisson--Nijenhuis manifold $(M, \pi, N)$, there is a unique,
 up to isomorphisms,
 symplectic--Nijenhuis local groupoid whose induced
Poisson--Nijenhuis structure on the unit space is $(M, \pi, N)$.
\end{itemize} 
In other words, there is a one-to-one correspondence between
Poisson--Nijenhuis manifolds and symplectic-Nijenhuis local 
 groupoids.
\end{theorem}
We will  sketch a proof of this theorem
since  we will need some intermediate
 results for our argument  later on 
(see Proposition \ref{pro:Zurich}),
which seem to not have appeared in the literature.

\begin{proof}[Proof of Theorem  \ref{thm:symplectic-nijenhuis-groupoids}]
To prove $(i)$, let $(\rgpd\toto M,\bomega,\bN)$ be a
 symplectic-Nijenhuis
local groupoid. Let  $\bpi$ be the Poisson bivector field on $\rgpd$ which
is the inverse of $\bomega$.
The pair $(\rgpd\toto M, \bomega)$ is a symplectic local groupoid.
  It is standard \cite{MR996653, MR866024}  that
 the pushforward 
\begin{equation}
\label{eq:1}
\pi:=\alpha_*\bpi 
\end{equation}
is a well defined Poisson  bivector field on $M$, and that
 the Lie algebroid of $\rgpd\toto M$ is isomorphic to the cotangent Lie algebroid $(T^*M)_\pi$ of
  $(M, \pi)$.

Now, let $\bpi_{\bN}\in\Gamma(\wedge^2T\rgpd)$ be the bivector field on
 $\rgpd$ as defined   by  \eq{eq:6.2}.
 Then the pair $(\rgpd \toto M, \bpi_{\bN})$ is a Poisson local groupoid. 
Analogous to Eq. \eqref{eq:1}, set
\begin{equation}
\label{eq:2}
\pi':=\alpha_* \bpi_{\bN}.
\end{equation}
 Then $\pi'$ is a well defined Poisson bivector field on $M$ as
well \cite{MR959095}. 
Finally, from Proposition \ref{thm:alt-pn-breakdown},
it follows that  the Schouten bracket $[\bpi, \bpi_{\bN}]$ vanishes, and
 thus we  have  $[\pi,\pi']=0$.

On the other hand, the Lie groupoid morphism $\bN :T\rgpd \to T\rgpd$
as in  \eqref{eq:Paris}
 induces a map $N=\bN|_{\epsilon_*(TM)}: TM \to TM$
 on its unit space, which is, clearly, a $(1,1)$-tensor. The Nijenhuis torsion free
 condition for $N$ then follows from that of $\bN$.
Moreover, it is clear that 
\begin{equation}
\label{eq:3}
\pi'^\sharp = N \smalcirc \pi^\sharp.
\end{equation}
Thus  $(M,\pi, N)$ is indeed a Poisson--Nijenhuis manifold, as desired.


Conversely, to see $(ii)$, let $(M,\pi,N)$ be a Poisson--Nijenhuis manifold.
 Let  $A=(T^*M)_{\pi}$ be the cotangent Lie algebroid of $(M, \pi)$,
 and $\bPl(A)\toto M$  be
the corresponding local Lie groupoid as in
  Theorem \ref{thm:local-groupoid}. Let $\bomega$ be the multiplicative
 symplectic form on $\bPl(A)$ 
as in  Theorem \ref{thm:cattaneo-felder} (ii),
 and $\bpi$ be its associated Poisson
 bivector field. Then, under the correspondence between Lie bialgebroids and Poisson local groupoids spelled out in Theorem \ref{thm:mackenzie-xu}, the Poisson  groupoid $(\bPl(A)\toto M, \bar\pi)$ is associated to the Lie bialgebroid $((T^*M)_\pi, TM)$.

 On the other hand, out of the same Poisson--Nijenhuis structure
 $(M,\pi,N)$, we can construct   another natural Lie bialgebroid
 $((T^*M)_\pi, (TM)_N)$ \cite{MR1421686}.
Here, the  Lie algebroid $(T^*M)_\pi$ is, as usual, the cotangent 
 Lie algebroid of $(M,\pi)$. The  Lie algebroid $(TM)_N$
 consists of the triple $(TM, \rho_N, [\cdot,\cdot]_N)$ defined as follows. The underlying vector bundle is the tangent bundle $TM$ of $M$, while the anchor $\rho_N$ and bracket $[\cdot, \cdot]_N$ are given, respectively, by
\begin{align*}
  \rho_N(X) & = NX,  \quad \text{and}\\
  [X,Y]_N &= [NX, Y] + [X,NY] - N[X,Y], \ \ \ \forall X, Y \in \Gamma (TM).
\end{align*}
One should think of $(TM)_N$ as a twisted version of the tangent Lie algebroid
 $TM$,
 whose twist is given by the Nijenhuis tensor $N$.

To this  Lie bialgebroid $((T^*M)_\pi, (TM)_N)$, we can apply
 Theorem \ref{thm:mackenzie-xu}   to obtain
 a second natural   
multiplicative Poisson bivector field $\bpi'$
 on $\bPl(A)$, which makes the pair $(\bPl(A)\toto M, \bpi')$ 
into a Poisson local groupoid.

To complete the proof of part $(ii)$,  it  remains to prove
 that the two multiplicative Poisson structures $\bpi$ and $\bpi'$
 on $\bPl(A)$ satisfy  the condition:
 \begin{equation}\label{part-ii-vanishing-bracket}
   [\bpi, \bpi'] = 0.
 \end{equation}
 Indeed, assuming Eq. \eqref{part-ii-vanishing-bracket} holds, let 
 \begin{equation}\label{eq:bar-n}
   \bN=(\bpi')\diese\circ\bomega\bemol:T\bPl(A)\to T\bPl(A).
 \end{equation}
 From Proposition \ref{pro:Vaisman}, it follows
 that $\bN$ is indeed  a Nijenhuis tensor,
 and therefore   $(\bPl(A), \bpi, \bN)$ is a   Poisson-Nijenhuis manifold.
Moreover $\bN$ is a multiplicative $(1, 1)$-tensor. Thus it follows  
that  $(\bPl(A)\toto M,\bomega,\bN)$ is a symplectic-Nijenhuis local groupoid. 

In order to prove Eq.\ \eqref{part-ii-vanishing-bracket}, let
 $\delta: \Gamma (\wedge^\bullet T^*M)\to \Gamma (\wedge^{\bullet+1} T^*M)$
be the \Liealgebroid of the Lie algebroid $(TM)_N$ 
(see Eq.\ \eqref{eq:BRU}). Then  one  can easily  check that
\begin{equation}\label{part-ii-compatibility-derham}
  [\delta, d_{\text{DR}}]=0,
\end{equation}
where $d_{\text{DR}}$ is the De Rham  differential operator on
$\Gamma(\wedge^\bullet T^*M) = \Omega^\bullet(M)$ 
\cite[Lemma 5.3]{MR2276462}. 
%
It is  well known that
when $A=TM$ is the tangent Lie algebroid of a manifold $M$,
its \Liealgebroid  is the De Rham differential
 operator  $d_{\text{DR}}$. According to the  Universal Lifting
 Theorem \cite{MR2911881}, we see that \eq{part-ii-compatibility-derham}
implies \eq{part-ii-vanishing-bracket}.

Finally,  it is simple to check that the two constructions we spelled out in showing parts $(i)$ and $(ii)$ 
are  indeed inverse to each other. This  concludes the proof of Theorem \ref{thm:symplectic-nijenhuis-groupoids}.
\end{proof}

Let us single out the following important fact that  we will
need later on.

\begin{proposition} 
\label{pro:Zurich}
Let $(\rgpd\toto M,\bomega,\bN)$ be a symplectic-Nijenhuis local
 groupoid with the induced  Poisson--Nijenhuis structure
 $(M, \pi, N)$ on its unit space as in Theorem  \ref{thm:symplectic-nijenhuis-groupoids}.
 Then the source map $\alpha: \rgpd\to M$ is a Poisson--Nijenhuis map. 
In particular, we have
$$\alpha_* \bpi=\pi, \ \alpha_*\bpi_{\bN}=\pi_N, \ \alpha_* \smalcirc \bN  
= N\smalcirc \alpha_*, $$    
where $\bpi$  denotes  the bivector field on $\rgpd$ inverse to  $\bomega$.
\end{proposition}
\begin{proof}
The first identity is exactly Eq. \eqref{eq:1}.
The second identity follows  from Eqs. \eqref{eq:2}-\eqref{eq:3}.
Finally, the last identity  is a consequence  of 
the fact that $\bN: T\Sigma\to T\Sigma$ is a groupoid morphism
and therefore commutes with the source map $\alpha_*: T\Sigma\to TM$
---see  \eqref{eq:Paris}. 
\end{proof}

\section{The Complete Lift to The Cotangent Bundle}

We start by recalling the definition of the complete lift
of  $(1, 1)$-tensors to the cotangent bundle
 and some related standard facts.  For details, we refer the readers
to  \cite{MR1021489} and  its references.

Let $N : TM \to TM$ be a $(1,1)$-tensor on a manifold $M$. Denote
 by $$\langle \cdot, \cdot \rangle : T^*M \times_M TM \to M \times \Rr$$ the canonical pairing. 
There is a natural $1$-form $\theta_N \in \Omega^1(T^*M)$ defined by
$$ \theta_N(u) = \langle \lambda, N( p_*(u)) \rangle $$
$\forall u\in T_\lambda (T^*M)$, where $p: T^*M \to M$ is the  canonical
 projection. In particular, if $N=\Id$,
 then $\theta_N$ is just the
 Liouville form on $T^*M$.

\begin{definition}
   The \emph{complete lift of $N$ to the cotangent bundle} is the
 $(1,1)$-tensor
 $$N^c : TT^*M \to TT^*M$$ on $T^*M$ defined by the property that
  \begin{equation}
\label{eq:Naples}
 \omega_{\text{can}}( N^c u, v ) = (d\theta_N)(u,  v), 
\end{equation}
  for any $\lambda\in T^*M$ and any $u,v\in T_\lambda (T^*M)$.
Here $\omega_\can \in \Omega^2 (T^*M)$
is  the canonical symplectic form on $T^*M$.
\end{definition}

It can be checked, by a direct computation, that
\begin{equation}\label{eq:complete-lift-to-tstar}
  \omega_\can(N^c u,v) = \omega_\can((\transpose N)_* u, (\transpose N)_* v),
\end{equation}
 $\forall \lambda\in T^*M, \ u,v\in T_\lambda (T^*M)$. 
  Here $N^T : T^*M \to T^*M$ is the  dual of $N$,
 and $(N^T)_* : TT^*M \to TT^*M$ denotes its tangent map.

\begin{lemma}\label{lem:lie-poisson-equals-2nd-poisson-pn}
%
Let $N:TM\to TM$ be a  Nijenhuis tensor on $M$.
Denote by $\pi'\in \XX^2(T^*M)$ the Poisson bivector
field on  $T^*M$ of the Lie--Poisson structure
corresponding to  the Lie algebroid $(TM)_N$. Then
\begin{equation}
\label{eq:Rome}
   (\pi')^\sharp \smalcirc \omega_\can \bemol = N^c. 
\end{equation}
\end{lemma}
\begin{pf}
  For any $X \in \XX (M)$, let $\lin{X}\in C^\infty   (T^*M)$
 be the fibrewise linear function on $T^*M$ defined by
  $$\lin{X}(\lambda) = \langle \lambda, X_x \rangle,$$
  for any $\lambda \in T^*_xM$ $(x \in M)$.
  By definition, for any  $X,Y\in \XX(M)$ and any 
 $f, g \in \smooth(M)$, we have:
  \begin{align}
    &\{ \lin{X}, \lin{Y} \}_{\pi'} = \lin{[N(X),Y] + [X,N(Y)] - N([X,Y])},\label{eq:tpiN-1}\\
    &\{\lin{X}, p^*f\}_{\pi'} = p^*\langle df, N(X)\rangle\label{eq:tpiN-2},\\
    &\{p^*f, p^*g\}_{\pi'} = 0\label{eq:tpiN-3}.
  \end{align}
  For any given $\psi\in \smooth(T^*M)$, we denote by $\ham{\psi}\in \XX(T^*M)$ the Hamiltonian
 vector field of $\psi$ with respect to  the canonical  symplectic
 structure on $T^*M$, 
i.e.\ $\ham{\psi} = \pi_\can^\sharp(d\psi)$, where 
$\pi_\can^\sharp = (\omega_\can \bemol)^{-1}$.
  Note that Eq. \eqref{eq:Rome} is equivalent to
\begin{equation}
\label{eq:Athens}
 N^c\smalcirc \pi_\can^\sharp=  (\pi')^\sharp  .
\end{equation}
The latter is equivalent to 
\begin{equation}
\label{eq:Padova}
N^c(\ham{F})(G) =  (\pi')^\sharp(dF)(G)
\end{equation}
for any $F, G\in  C^\infty (T^*M)$.

From Eq. \eqref{eq:Naples}, it follows that
 $\omega_\can \bemol \smalcirc N^c$ is skew-symmetric.
Since $\pi_\can^\sharp = (\omega_\can \bemol )^{-1}$,
a simple linear algebra argument implies that
$N^c\smalcirc \pi_\can^\sharp =N^c\smalcirc (\omega_\can \bemol )^{-1}$ is
 also skew-symmetric. 
Note that $C^\infty (T^*M)$ is spanned locally by two types of
 functions of the   form  $\lin{X}$ and $p^*f$,  $\forall X\in \XX (M)$ and 
$f \in C^\infty (M)$. In order to prove Eq.  \eqref{eq:Padova},
 it thus suffices
to prove the following identities:
\begin{align}
    &N^c(\ham{p^*f}) (p^*g) = (\pi')^\sharp(d(p^*f))(p^*g), \ \ \forall f, g
 \in C^\infty (M)  \label{eq:Geneva}\\
    &N^c(\ham{\lin{X}}) (\lin{Y}) =  (\pi')^\sharp(d\lin{X}) (\lin{Y}) \ \ \forall X, Y\in \XX(M) \label{eq:Prugia},\\
    &N^c(\ham{\lin{X}}) (p^*f)=  (\pi')^\sharp(d\lin{X}) (p^*f), \ \  \forall f
\in C^\infty (M), \  X\in \XX(M)
\label{eq:Scale}.
  \end{align}

Since $p_* \big(\ham{p^*f}\big)=0$, it follows from the definition of  the
complete lift $N^c$ that $p_*\big(N^c(\ham{p^*f}\big)=0$.
Hence both sides of Eq. \eqref{eq:Geneva} vanish.

Now we  prove Eq. \eqref{eq:Prugia}.
 Note 	 that the following relation
 is standard \cite[Proposition 5.4.3]{MR1021489}:

 \begin{equation}\label{eq:standard-completelift}
   N^c(\ham{\lin{X}}) - \ham{\lin{N(X)}} = \pi_\can^\sharp(\theta_{L_XN}),
 \end{equation}
 where  $L_X N$ denotes the usual Lie derivative of the $(1,1)$-tensor
 $N$ given by 
$(L_XN)(Y) = [X,N(Y)]-N([X,Y]), \ \forall X, Y\in \XX(M)$. 
  Also, the following identity  can be   verified directly:
  \begin{equation}\label{eq:standard-liederiv}
    [\pi_\can^\sharp(\theta_{L_XN})](\lin{Y}) = \lin{(L_XN)(Y)}.
  \end{equation}
  From  \eq{eq:standard-completelift}-\eq{eq:standard-liederiv},
it thus follows that
  \begin{align*}
    [N^c(\ham{\lin{X}})](\lin{Y}) & = \ham{\lin{N(X)}}(\lin{Y}) + (\pi_\can^\sharp(\theta_{L_XN}))(\lin{Y}) \\
    & = \lin{[N(X),Y]} + \lin{(L_XN)(Y)} \\
    & = \lin{[N(X),Y] + [X,N(Y)] - N([X,Y])}.
    \numberthis\label{eq:compat-linlin}
  \end{align*}
Now Eq. \eqref{eq:Prugia} follows from combining 
 Eq.\ (\ref{eq:compat-linlin}) with Eq.\ (\ref{eq:tpiN-1}).

Finally, we prove  Eq. \eqref{eq:Scale}.
  First,   we  have $[\pi_\can^\sharp(\theta_{L_XN})](p^*f) = 0$.
 Now, applying \eq{eq:standard-completelift}
 to the  function  $p^*f$, we   have
  \begin{align*}
    [N^c(\ham{\lin{X}})](p^*f) & = \ham{\lin{N(X)}}(p^*f) +
 (\pi_\can^\sharp(\theta_{L_XN}))(p^*f)\\
    &= p^* \big( (NX)( f) \big)\\
&=(\pi')^\sharp(d\lin{X}) (p^*f).
 \numberthis \label{eq:compat-linpullback}
  \end{align*}
This concludes the proof of  the lemma.
\end{pf}

\begin{remark}
According to a theorem of Vaisman \cite{MR1345608},
the Poisson bivector field $\pi'$ is compatible with the canonical Poisson structure $\pi_\can$ on $T^*M$ in the sense that the Schouten bracket $[\pi', \pi_\can]$ vanishes. 
Therefore, from Lemma \ref{lem:lie-poisson-equals-2nd-poisson-pn}, it 
follows that $(T^*M, \pi_\can, N^c)$ is a Poisson--Nijenhuis structure
 on $T^*M$, and  moreover its \emph{second} Poisson tensor $\pi_{N^c}$
 defined, as usual, by
\begin{equation}\label{eq:poisson-cotangent-from-pn}
  \pi_{N^c}^\sharp = N^c\smalcirc\pi_\can^\sharp 
\end{equation}
coincides with the Lie--Poisson structure of the Lie algebroid 
$(TM)_N$.
\end{remark}

We now recall the following well-known fact from the general theory of Poisson groupoids \cite{MR1262213,MR1746902}. For any Poisson local groupoid $(\rgpd\toto M,\bpi)$ with Lie bialgebroid
 $(A ,A^*)$, the  following diagram of vector bundle morphisms:
  \begin{equation}\label{eq:poisson-groupoid-to-bialgebroid-diagram}
    \begin{gathered}
      \begin{tikzpicture}
        \matrix (m) [matrix of math nodes, row sep=3em, column sep=8em] {
          \Lie(T^*\rgpd) & \Lie(T\rgpd) \\
          T^*A & TA \\
        };
        \path[-stealth]
        (m-1-1) edge node [above] {$\Lie(\bpi^\sharp)$} (m-1-2)
        (m-1-1) edge node [left] {$\flippypairy_\rgpd$} (m-2-1)
        (m-1-2) edge node [right] {$\flippy_\rgpd$} (m-2-2)
        (m-2-1) edge node [below] {$\pi^\sharp_{A}$} (m-2-2);
      \end{tikzpicture}
    \end{gathered}
  \end{equation}
  commutes. 
Here $\pi_{A}$ denotes the Lie--Poisson structure on $A$ induced by the Lie 
algebroid structure on $A^*$,
 and $\flippypairy_\rgpd, \flippy_\rgpd$ are 
 natural vector bundle isomorphisms which we shall not make explicit for
the sake of brevity. See  \cite{MR1262213, MR1746902} for more details.

Let $\bN: T\rgpd \to T\rgpd$ be a multiplicative $(1, 1)$-tensor on $\rgpd$. There is an associated tensor $TA\to TA$ on $A$,
which we will denote by $\LLie(\bN)$ and call the 
\emph{infinitesimal of $\bN$} following \cite{MR2545872}. The relationship between the infinitesimal of $\bN$ and the image of $N$ under the usual 
Lie functor is given  by the simple identity 
\begin{equation}
\label{eq:Lie}
\LLie(\bN) = \flippy_\rgpd \smalcirc \Lie(\bN) \smalcirc \flippy_\rgpd^{-1}.
\end{equation}

Let $(\rgpd\toto M, \bomega,\bN)$ be a symplectic-Nijenhuis local
 groupoid with induced Poisson--Nijenhuis structure $(M,\pi,N)$ on the unit
 space $M$.
 We can now show that the infinitesimal of the Nijenhuis tensor $\bN$ coincides with the complete lift $N^c: TT^*M \to TT^*M$ of $N$.

\begin{proposition}\label{lem:lie-lifted-nijenhuis}
  Let $(\rgpd\toto M, \bomega, \bN)$ be a symplectic-Nijenhuis local
  groupoid, and $(M, \pi, N)$ the corresponding
Poisson--Nijenhuis structure on $M$, as in
Theorem \ref{thm:symplectic-nijenhuis-groupoids}.
Then 
\begin{equation}
\label{eq:Orly}
 \LLie(\bN) = N^c .
\end{equation}
\end{proposition}
\begin{proof}
  By definition, $\bN = \bpi_\bN^\sharp \smalcirc \bomega\bemol$,
 and hence
  $$\LLie(\bN) = \flippy_\rgpd \smalcirc \Lie(\bpi_\bN^\sharp) \smalcirc \Lie(\bomega\bemol) \smalcirc\flippy_\rgpd^{-1}.$$
  We now check that
\begin{equation}
\label{eq:SCE}
\Lie(\bomega\bemol) = (\flippypairy_\rgpd)^{-1} \smalcirc \omega_\can\bemol \smalcirc \flippy_\rgpd.
\end{equation} 

Let $\bpi$ be the Poisson bivector field on $\rgpd$
 inverse to $\bomega$. Then $(\rgpd\toto M, \bpi)$ is a Poisson groupoid,
 and, by Theorem \ref{thm:mackenzie-xu}, its Lie bialgebroid is
 $((T^*M)_\pi,TM)$. Since the Lie--Poisson structure
induced by the tangent bundle Lie algebroid $TM$ coincides with the 
canonical symplectic structure on $T^*M$, the commutativity of 
diagram (\ref{eq:poisson-groupoid-to-bialgebroid-diagram}) implies that
 $$\Lie(\bpi^\sharp)= \flippy_\rgpd^{-1} \smalcirc 
\pi_\can^\sharp \smalcirc \flippypairy_\rgpd.$$
 The latter is equivalent to Eq. \eqref{eq:SCE}, as claimed.

Finally, recall that the Lie bialgebroid of
the Poisson groupoid
 $(\rgpd\toto M, \bpi_\bN)$ is isomorphic to $((T^*M)_{\pi}, (TM)_N)$ according to the proof of
Theorem \ref{thm:symplectic-nijenhuis-groupoids}.
 From  commutative  diagram (\ref{eq:poisson-groupoid-to-bialgebroid-diagram}), 
 we thus have 
$$\Lie(\bpi^\sharp_\bN) =
 \flippy_\rgpd^{-1} \smalcirc (\pi')^\sharp\smalcirc \flippypairy_\rgpd.$$
Here $\pi'$ is the Lie--Poisson structure on $T^*M$ corresponding
to the Lie algebroid $(TM)_N$   as in 
Lemma \ref{lem:lie-poisson-equals-2nd-poisson-pn}.
In particular, we have
$$ \LLie(\bN) = \flippy_\rgpd \smalcirc \Lie(\bpi_\bN^\sharp)
 \smalcirc \Lie(\bomega\bemol) \smalcirc\flippy_\rgpd^{-1} =
 (\pi')^\sharp \smalcirc \omega_\can\bemol = N^c. $$
Here we used Eq.\ \eqref{eq:SCE} for the second equality, and
 Lemma \ref{lem:lie-poisson-equals-2nd-poisson-pn} 
for the last equality. This concludes the proof.
\end{proof}

\section{Symplectic Realizations of Non-Degenerate Poisson--Nijenhuis Manifolds}

In this section,
 we conclude the proof of a more general version of 
Theorem \ref{thm:symplectic-realization-holomorphic} that holds in the case
 of non-degenerate Poisson--Nijenhuis structures. The main result here is Theorem \ref{thm:symplectic-realization-PN}. 
According to  Lemma \ref{lem:hp-is-pn}, it is clear that this includes the case of holomorphic Poisson manifolds.

The following standard lemma is  crucial to our proof.

\begin{lemma}[\cite{MR1421686}]\label{lem:bialgebroid-isomorphism}
  Let $(M,\pi,N)$ be a Poisson-Nijenhuis manifold and $\pi_N$  its second
 Poisson structure,
 as in Proposition \ref{prop:second-poisson-pn}, defined by Eq. \eqref{eq:SCE}.
 The pair of maps 
  $$\transpose N: (T^*M)_{\pi_N}\to (T^*M)_\pi, \;  N: (TM)_N \to TM$$
  defines a morphism of Lie bialgebroids 
  $$(\transpose N, N): ((T^*M)_{\pi_N}, TM) \to ((T^*M)_{\pi}, (TM)_N).$$
   In particular, if $N$ is invertible, the pair $(\transpose N,N)$ is an isomorphism of Lie bialgebroids.
\end{lemma}

Let us assume, throughout the remainder of this section, that we are given a Poisson--Nijenhuis manifold $(M, \pi, N)$ whose Nijenhuis tensor $N : TM \to TM$ is invertible. In 
this case, we also say that $(M, \pi, N)$ is a 
\emph{non-degenerate Poisson-Nijenhuis manifold}. Also, in order  to simplify notation,
 we denote by $A$ (resp.\ $A_N$)  the cotangent Lie algebroid $(T^*M)_\pi$
 (resp.\ $(T^*M)_{\pi_N}$) of $\pi$ (resp.\ $\pi_N$).

The pair $(\bPl(A_N)\toto M, \bomega')$ is a symplectic local groupoid,
 where $\bomega'$ is the symplectic form  as in
 Theorem \ref{thm:cattaneo-felder}. Let $\bpi'$ be the
corresponding Poisson structure on $\bPl(A_N)$, which is 
 the inverse of $\bomega'$. Then the Lie bialgebroid of the Poisson local groupoid $(\bPl(A_N)\toto M,\bpi')$ is $(A_N, TM)$.

On the other hand, we can construct another Poisson local
 groupoid $(\bPl(A)\toto M, \bpi_\bN)$ as follows.  
According to  Theorem \ref{thm:symplectic-nijenhuis-groupoids},
there is a  symplectic-Nijenhuis local groupoid 
 $(\bPl(A)\toto M, \bomega, \bN)$,
 which induces the Poisson--Nijenhuis structure $(M,\pi,N)$ on the unit space $M$. Let $\bpi$ be the Poisson bivector field associated to $\bomega$,
 and let $\bpi_\bN$ be defined, as before, by the relation
 $\bpi_\bN^\sharp = \bN \smalcirc \bpi^\sharp$. Then the Poisson local groupoid
 $(\bPl(A)\toto M, \bpi_\bN)$ has Lie bialgebroid $((T^*M)_\pi, (TM)_N)$.

Now, according to Lemma \ref{lem:bialgebroid-isomorphism},
we have a Lie bialgebroid  morphism
\begin{equation}\label{eq:nt-n}
  (\transpose N, N): ((T^*M)_{\pi_N}, TM) \to ((T^*M)_{\pi}, (TM)_N).
\end{equation}
Thus, from Theorem \ref{thm:mackenzie-xu}, it follows that the induced morphism
of local Lie groupoids  
\begin{equation}\label{eq:bpnt}
  \bP(\transpose N) : \bPl(A_N) \to \bP(A),
\end{equation}
as in Proposition \ref{prop:paths-to-groupoid} (i),
 is a Poisson map. Hence we have
\begin{equation}\label{eq:poisson-iso}
  \bP(\transpose N)_* \bpi' =\bpi_\bN.
\end{equation}
Since $N$ is invertible by assumption, the map in \eqref{eq:nt-n} is 
an isomorphism of Lie bialgebroids.
Therefore the map in \eqref{eq:bpnt} is
indeed an isomorphism of  Poisson local groupoids. 
In particular, the bivector field
 $\bpi_\bN$ is non-degenerate, since $\bpi'$ is non-degenerate.
 
Let $\bomega_\bN$ be the (necessarily multiplicative) symplectic form on $\bPl(A)$
whose Poisson bivector field is $\bpi_\bN$.
 Then $(\bPl(A)\toto M, \bomega_\bN)$ is a symplectic
local groupoid.  
Since $\bP(\transpose N)$ is a Poisson isomorphism, 
 we must have
\begin{equation}\label{eq:bomega-bomegaprime}
 \bomega_\bN = (\bP(\transpose N)^{-1})^*\bomega'.
\end{equation}

Summarizing, we have proved the following

\begin{proposition}\label{prop:summary-intermediate-PN}
  The pair $(\bPl(A)\toto M, \bar\omega_\bN)$ is a symplectic local groupoid 
which,  as a Poisson groupoid, has Lie bialgebroid $((T^*M)_\pi,(TM)_N)$.
\end{proposition}

Now fix $\nabla$ an $A$-connection on $A$,  and let $U\subset T^*M$ be
a  sufficiently small open neighborhood
 around the zero section, as in Proposition \ref{thm:exp-local-diffeo}. 
Define
\begin{align}\label{eq:pullback-omegaNbar}
  \begin{gathered}
    \pomega = \exp^*\,\bomega, \quad \text{and}\\
    \pomega_N = \exp^*\,\bomega_\bN,
  \end{gathered}
\end{align}
where $\exp : U \to \bPl(A)$   is
 the groupoid exponential map  associated to $\nabla$. The formula 
 \eqref{eq:pullback-omegabar-formula} still holds for $\pomega$, 
 since $(\bPl(A)\toto M, \bomega)$ is exactly the same symplectic local
 groupoid as in Theorem \ref{thm:cattaneo-felder}.
On the other hand,  we also  have 

\begin{proposition}
\label{pro:Kaifeng}
The symplectic form $\pomega_N\in \Omega^2(U)$  can be explicitly expressed
as follows:
\begin{equation}\label{eq:pullback-omegaNbar-formula}
  \pomega_N = \int_0^1 ((\transpose N)^{-1}\smalcirc \varphi^\axi_t)^*\omega_\can \; dt,
\end{equation}
where  $\omega_\can \in \Omega^2 (T^*M)$
is the canonical symplectic form on $T^*M$, and 
	$\varphi^\axi_t$ is the flow of the geodesic vector field of $\nabla$.
\end{proposition}

\begin{proof}
  By Proposition \ref{prop:paths-to-groupoid} (i),
 we have
  $$ \P(\transpose{(N^{-1})})^*\smalcirc q'^* = q^*\smalcirc \bP((\transpose N)^{-1})^* $$
  where $q : \Pl(A) \to \bPl(A)$ and $q' : \Pl(A_N) \to \bPl(A_N)$ are the 
 quotient maps as in \eqref{eq:q}.
 It is also simple to see that $\P((\transpose N)^{-1}) = (\P(\transpose N))^{-1}$,
 and  $\bP( (\transpose N)^{-1} ) = (\bP (\transpose N))^{-1}$. Since $q'^* \bomega' = \iota'^* \tomega_\can$, we have 
  $$q^* (\bP(\transpose N)^{-1})^* \bomega' =  P(\transpose{(N^{-1})})^*q'^* \bomega' = P(\transpose{(N^{-1})})^* \iota'^*\tomega_\can.$$
By commutative diagram \eqref{eq:leaf-space-exp}, we have
  $\exp^* = \Phi^* q^*$. Thus it follows that
  $$ \pomega_N = \exp^* (\bP(\transpose N)^{-1})^* \bomega' = \Phi^* q^* (\bP(\transpose N)^{-1})^* \bomega' = \Phi^* P(\transpose{(N^{-1})})^* \iota'^* \tomega_{\can}. $$ 
  Now  $\forall \  \la\in U$ and $u,v\in T_\la A$, we have
  \begin{align*}
    \pomega_N(u,v) & = 
\tilde\omega_{\can}(\iota'_* P((\transpose N)^{-1})_* \Phi_* u, \ \  \iota'_* P((\transpose N)^{-1})_* \Phi_* v) \\ 
    & = \int_0^1 \omega_{\can}((\tau [\iota'_* P((\transpose N)^{-1})_* \Phi_* u])(t), (\tau [\iota'_* P((\transpose N)^{-1})_* \Phi_* v])(t)) dt\\
    & = \int_0^1 \omega_{\can}( ((\transpose N)^{-1})_* (\varphi^\nabla_t)_* u, ((\transpose N)^{-1})_* (\varphi^\nabla_t)_* v ) dt \\
    & = \left(\int_0^1 ((\transpose N)^{-1} \smalcirc \varphi^\nabla_t)^* \omega_{\can}dt\right)(u,v),
  \end{align*}
where the second to last equality follows from \eq{eq:zaventem}.
\end{proof}

Combining Proposition \ref{prop:paris},
Proposition \ref{pro:Zurich} and Proposition \ref{pro:Kaifeng}, we
are finally led to  the following  main theorem of this section.

\begin{theorem}\label{thm:symplectic-realization-PN}
Let $(M,\pi, N)$ be a non-degenerate  Poisson--Nijenhuis manifold,
and $A=(T^*M)_\pi$ the cotangent Lie algebroid of the Poisson manifold
$(M, \pi)$.
  Fix $\nabla$ an $A$-connection on $A$ and let $\varphi^\nabla_t$ be the flow of the geodesic vector field of $\nabla$.
  Also let $U\subset A$ be a sufficiently small open neighborhood of the 
zero section of  $A$
as in Proposition \ref{thm:exp-local-diffeo}.
Then the following assertions hold.
\begin{itemize}
  \item [(i)]
 The projection $\pr|_U: U\to M$, together with the symplectic form $\pomega$ (resp.\ $
\pomega_N$),
defined by \eq{eq:pullback-omegabar-formula} 
(resp.\ \eq{eq:pullback-omegaNbar-formula}),
is   a symplectic realization of $\pi$ (resp.\ $\pi_N$).
\item [(ii)]   The $(1, 1)$-tensor 
  \begin{equation}\label{eq:pN}
    \pN:=(\pomega_N\bemol)^{-1} \smalcirc\pomega\bemol:TU\to TU
  \end{equation}
is a Nijenhuis tensor on $U$. Furthermore, the triple $(U, \pomega, \pN)$ is
a symplectic-Nijenhuis manifold. 
\item[(iii)] Denote by $\ppi$ the Poisson bivector field
 inverse to $\pomega$. Then the canonical projection $\pr|_U: U \to M$ is
 a Poisson--Nijenhuis map with respect to $(U, \ppi, \pN)$  and $(M, \pi, N)$.
\item  [(iv)] The zero section is a Lagrangian submanifold of $U$ with
respect to both $\pomega$ and $\pomega_N$.
\end{itemize}
\end{theorem}

Following Petalidou \cite{2015arXiv150107830P}, 
we will call  the symplectic-Nijenhuis manifold $(U, \pomega, \pN)$,
 endowed with the  projection $\pr|_U:U\to M$, a 
{\em symplectic realization of the non-degenerate Poisson--Nijenhuis 
manifold} $(M, \pi, N)$.  Note that a symplectic realization
of a  Poisson--Nijenhuis manifold can only exist when the
Nijenhuis tensor is invertible.

\begin{remark}
As an immediate consequence of \eq{eq:complete-lift-to-tstar}, 
Eq.  \eqref{eq:pullback-omegaNbar-formula} can be rewritten as
\begin{equation}
\label{eq:Fani}
  \pomega_N(u,v) = \int_0^1 (\varphi^\axi_t)^*\omega_{\text{can}}( (N^c)^{-1}u,v) \; dt
\end{equation}
$\forall \xi \in U$ and  $u,v\in T_{\xi}(T^*M)$. The explicit formula
 of Eq. \eqref{eq:Fani} is due to Petalidou \cite{2015arXiv150107830P}.
In fact, Theorem \ref{thm:symplectic-realization-PN} (i)-(ii)
 essentially
recovers a theorem  claimed by Petalidou  \cite{2015arXiv150107830P},
 which was obtained  by following closely the computational approach
 in \cite{MR2900786}. From the
discussion of this section, we see that 
both symplectic forms $\pomega$ and $\pomega_N$ are
 in fact  conceptually parts of the data involved in constructing 
 a (\emph{a priori} hidden) symplectic-Nijenhuis local groupoid. 
\end{remark}

\begin{remark}
Theorem~\ref{thm:symplectic-realization-PN}
was reproved using a different method 
in the preprint \cite{alej2018local}, 
which appeared after the present paper was posted on arXiv. 
See~\cite[Section 3.2]{alej2018local} for details. 
We also refer the reader to~\cite{MR2116732}
for results closely related to those in~\cite{alej2018local}.
\end{remark}
\color{black}

\section{Holomorphic Symplectic Realizations of Holomorphic Poisson Manifolds}
\label{section:PEK}

It remains to explain how Theorem \ref{thm:symplectic-realization-holomorphic} follows from Theorem \ref{thm:symplectic-realization-PN}. First, recall the following standard fact.

\begin{proposition}[\cite{MR2439547}]\label{prop:factor-4}
  Let $(X,\omega)$, where $\omega= \omega_R + i\omega_I \in \Omega^{2, 0}(X)$,
 be a holomorphic symplectic manifold. Denote by
 $\pi = \pi_R + i\pi_I \in \Gamma (\wedge^2 T^{1, 0}X)$  the associated holomorphic Poisson 
bivector field. Then the  real differential $2$-forms $\omega_R$
 and  $\omega_I \in \Omega^2(X)$ are
 symplectic,
 and their corresponding Poisson bivector fields
are  $4\pi_R$ and $-4\pi_I$,  respectively.
\end{proposition}

We are now ready to conclude the proof of the main theorem of this paper.

\begin{proof} [Proof of Theorem \ref{thm:symplectic-realization-holomorphic}]
Let $(\holom{X}, \pi=\pi_R + i\pi_I)$ be a holomorphic Poisson manifold with almost complex structure $J$ and the underlying real manifold $X$. By Lemma \ref{lem:hp-is-pn}, 
 $(X,\pi_I,J)$ is a Poisson--Nijenhuis manifold. Let $A=(T^*X)_{\pi_I}$ be the cotangent Lie algebroid of the real Poisson manifold $(X,\pi_I)$. 
Also fix an affine connection $\nabla^{TX}$ on $X$, 
and denote by $\nabla^{T^*X}$ the induced
linear  connection on $T^*X$.
 Finally, let $\nabla$ be the $A$-connection on $A$
as in  Example \ref{ex:brussel}, which is defined by
$$\nabla_a b = \nabla^{T^*X}_{\rho(a)}b\quad \forall a,b\in \Gamma(A),$$
where $\rho : A \to TX$ is the anchor of $A$.

From Theorem \ref{thm:symplectic-realization-PN} (i), it follows that there is
 an open neighborhood $Y\subset T^*X$ of the zero section where the 
 symplectic forms $\uomega_R$ and $\uomega_I$,
given, respectively,  by  \eq{eq:pullback-omegaNbar-formula} and
 \eq{eq:pullback-omegabar-formula},
 together with the projection $\pr|_Y : Y \to X$, give 
 symplectic realizations of $\pi_R$ and $\pi_I$, respectively. Furthermore,
 by Theorem \ref{thm:symplectic-realization-PN} (ii),  the $(1,1)$-tensor
\begin{equation}
\label{eq:J}
\uJ := (\uomega_R^\flat)^{-1} \smalcirc \uomega_I^\flat : TY \to TY
\end{equation}
is  Nijenhuis,
and  $(Y, \pomega_I, \pJ)$ is a symplectic-Nijenhuis manifold.
Moreover, the canonical projection  $\pr|_Y : Y \to X$ is a Poisson-Nijenhuis
 map
according to  Proposition \ref{pro:Zurich}.

We have the following lemma.

\begin{lemma}\label{prop:N-almost-complex}
  The $(1, 1)$-tensor $\uJ$ in \eqref{eq:J}
 is an almost complex structure on $Y$, i.e.\ $\uJ^2 = -1$.
\end{lemma}
\begin{proof}
  Recall that the $(1,1)$-tensor $\bJ: T\bPl(A) \to T\bPl(A)$, defined as in \eq{eq:bar-n}, is a local groupoid morphism with respect to the groupoid structure $\bPl(A)\toto M$ of Theorem \ref{thm:local-groupoid}. Also note that, by definition,
  $$\uJ = \exp_*^{-1} \smalcirc \bJ \smalcirc \exp_*.$$
  
On the other hand, we have $\LLie(\bJ) = J^c$ by 
Proposition \ref{lem:lie-lifted-nijenhuis}. Furthermore,
since $J$ is   an almost complex structure,
 it follows  (see \cite{MR1021489} for example) 
that  $(J^c)^2 = (J^2)^c = -1$. Thus $\LLie(\bJ^2)= (\LLie(\bJ))^{2} = -1$. Since  $\bJ$ is multiplicative on a local Lie groupoid,  it follows 
 that $\bJ^2 = -1$. 
Therefore  we have $\uJ^2 = -1$ as well. This  concludes the proof.
\end{proof}

Returning to the proof of Theorem \ref{thm:symplectic-realization-holomorphic}: since $\pJ$ is  already a  Nijenhuis tensor,  from Lemma
\ref{prop:N-almost-complex}, it follows that
 $ \pJ$  indeed induces  a complex structure on the  manifold $Y$, whose
underlying complex manifold is denoted by $\holom{Y}$.
Moreover, since $(Y, \pomega_I, \pJ)$ is a symplectic-Nijenhuis
manifold and its induced second Poisson structure is the
Poisson structure corresponding to $\pomega_R$, it follows that 
  $\pomega := \frac{1}{4}\left(\pomega_R - i \pomega_I\right) \in \Omega^2(Y)\otimes \CC$ yields a holomorphic symplectic form on $Y$ with
respect to the new complex structure $\pJ$ \cite{MR2439547}. 
  In particular, $(Y, \pomega, \pJ)$ is a holomorphic symplectic manifold. 
The triple $(Y, \pomega, \pJ)$   is indeed the underlying holomorphic  symplectic manifold of
the holomorphic symplectic local groupoid  integrating 
the given  holomorphic Poisson structure $\pi$ 
(see \cite[Theorem 3.22]{MR2545872} for an explanation of the factor
 $\frac{1}{4}$).
  Denote by $\ppi$ the associated holomorphic Poisson bivector field on $Y$. 
 From  Theorem \ref{thm:symplectic-realization-PN} (ii),
 Proposition \ref{prop:factor-4} and Proposition \ref{prop:pn-map},
it follows that
 the projection $\pr|_Y: Y \to X$ is  indeed
a holomorphic Poisson map with
respect to the  holomorphic Poisson structures 
$(Y, \ppi, \pJ)$ and $(X,\pi,J)$. This  concludes the proof.
\end{proof}

\appendix

\section{Lie bialgebroids and Poisson groupoids}

\begin{definition}\label{def:poisson-groupoid}
  A \emph{Poisson local groupoid} $(\rgpd\toto M, \bpi)$ is a local Lie groupoid $\rgpd\toto M$ endowed with a bivector field $\bpi \in \XX^2(\rgpd)$ such that the graph $\Lambda$ of multiplication in $\rgpd$:
$$ \Lambda \equiv \{ (x,y,x\cdot y) \mid (x,y)\in \rgpd\times \rgpd \text{ composable } \} \subset \rgpd \times \rgpd \times \bar\rgpd $$ 
is a coisotropic submanifold. Here $\bar\rgpd$ denotes $\rgpd$ endowed with the
  Poisson bivector field $-\bpi$.
\end{definition}

A bivector field $\bpi\in \XX^2(\rgpd)$ as in Definition \ref{def:poisson-groupoid} is also called multiplicative. In this context, the following is 
standard \cite{MR1262213}.

\begin{proposition}\label{prop:multiplicative-bivector}
  Let $A$ be the Lie algebroid of $\rgpd\toto M$. The bivector
field  $\bpi$ is multiplicative if and only if the map
 $\bpi^\sharp: T^*\rgpd \to T\rgpd$ is a local Lie groupoid morphism.
 Here $T^*\rgpd\toto A^*$ is the cotangent local Lie groupoid 
\cite{MR996653} and $T\rgpd\toto TM$ is the tangent local Lie groupoid of $\rgpd\toto M$.
\end{proposition}

\begin{definition}\label{def:symplectic-groupoid}
  A \emph{symplectic local groupoid} is a Poisson local groupoid $(\rgpd\toto M, \bpi)$ such that $\bpi$ is non-degenerate. 
\end{definition}

We now recall some fundamental facts regarding Lie bialgebroids. In 
the rest of this section, let $A$ be a Lie algebroid with anchor $\rho$ and
Lie  bracket $[\cdot,\cdot]$.

The Lie
 bracket, $[\cdot, \cdot]: \Gamma(A)\times \Gamma(A) \to \Gamma(A)$, can be 
extended to a bilinear bracket of multisections
 $\Gamma(\wedge^k A) \times \Gamma(\wedge^l A) \to \Gamma(\wedge^{k+l-1} A)$.
 We will denote both the initial bracket and its extension 
by  $[\cdot,\cdot]$. The triple $(\Gamma( \wedge^\bullet A), \wedge, [\cdot,\cdot])$ then forms a Gerstenhaber algebra
 \cite{MR1675117}.

Recall that Lie bialgebroids are a certain class of Lie algebroids $A$ 
for which the dual vector bundle $A^*$ also admits
 a compatible Lie algebroid structure. In order to define the compatibility condition, recall 
that the \emph{\Liealgebroid}
 of the Lie algebroid $A$ is the operator $d : \Gamma(\wedge^k A^*) \to \Gamma(\wedge^{k+1} A^*)$ defined by
\begin{align}
\label{eq:BRU}
 (d\lambda)( a_1 , \ldots, a_{k+1} ) & = \sum^{k+1}_{i=1} (-1)^{i+1} \rho(a_i)\cdot \lambda(a_1, \ldots, \hat a_i, \ldots, a_{k+1}) \\ & \quad\quad + \sum_{i<j} (-1)^{i+j} \lambda([a_i,a_j], a_1, \ldots, \hat a_i, \ldots, \hat a_j , \ldots, a_{k+1}).
\end{align}
for any $\lambda \in \Gamma(\wedge^k A^*)$, 
and any $a_1,\ldots, a_{k+1}\in \Gamma(A)$.

\begin{example}
  When $A=TM$ is the tangent Lie algebroid of a manifold $M$,
the Chevalley–Eilenberg  differential  $d$
coincides with  the De Rham differential operator  $d_{\text{DR}}$ 
 on $\Gamma(\wedge^\bullet T^*M) = \Omega^\bullet(M)$.
\end{example}

When $A^*$ happens to be a Lie algebroid as well, we denote by $d_* : \Gamma(\wedge^k A) \to \Gamma(\wedge^{k+1} A)$ the associated \Liealgebroid
 (acting on sections of $A \cong  (A^*)^*$). 

\begin{definition} [\cite{MR1262213, MR1421686}]
  Let $A$ be a Lie algebroid  such that $A^* $ also carries a Lie algebroid 
structure. Then $(A,A^*)$ is a \emph{Lie bialgebroid} if 
the Lie algebroid structures on $A$ and $A^* $
 are compatible in the following sense.
For any $a,a'\in \Gamma(A)$, one has
\begin{equation}\label{eq:bialgebroid}
  d_*[a,a'] = [d_*a,a'] + [a,d_*a'].
\end{equation}
\end{definition}

The compatibility condition (\ref{eq:bialgebroid}) is equivalent 
to asking that $d_*$ is a derivation of the Gerstenhaber algebra structure on $(\Gamma(\wedge^\bullet A), \wedge, [\cdot, \cdot])$ \cite{MR1675117}. 

\begin{example}
\label{ex:TSR}
  Let $A=TM$ be the tangent Lie algebroid and
 $A^*=(T^*M)_\pi$  the cotangent Lie algebroid of a Poisson manifold
 $(M,\pi)$. It is easy to see  that
$(A,A^*)$ is a Lie bialgebroid. In fact,
the graded Lie bracket  $[\cdot,\cdot] : \Gamma(\wedge^k A)\times
 \Gamma(\wedge^l A)\to \Gamma(\wedge^{k+1-1}  A)$ coincides with
 the Schouten bracket
 $[\cdot,\cdot]_S$ on $\Gamma(\wedge^\bullet TM)$, and 
 $d_* = [\pi,\cdot]_S$. Thus (\ref{eq:bialgebroid}) follows from the 
graded Jacobi identity of the Schouten brackets. 
\end{example}

\begin{definition}
  Let $(A,A^*)$ and $(B,B^*)$ be  Lie bialgebroids over
 the same base manifold $M$. A \emph{Lie bialgebroid morphism} 
$(\psi,  \transpose{\psi}) : (A,A^*) \to (B,B^*)$ is a vector bundle
 map $\psi : A \to B$ over the identity map such that
  \begin{itemize}
  \item[i.)] $\psi : A \to B$ is a Lie algebroid morphism, and
  \item[ii.)]  its dual $\transpose{\psi} : B^* \to A^*$ is also
 a Lie algebroid morphism.
  \end{itemize}
\end{definition}

One can prove that the definition of a Lie bialgebroid $(A,A^*)$ is symmetric
 in $A$ and $A^*$  \cite[Theorem 3.10]{MR1262213}.
 In particular, according to Example  \ref{ex:TSR}, we have

\begin{proposition}\label{prop:poisson-to-bialgebroid}
  Let $M$ be a Poisson manifold with Poisson bivector field $\pi \in \XX^2(M)$. Then $((T^*M)_\pi,TM)$ is a Lie bialgebroid.
\end{proposition}

The ``$d_*$'' operator of the Lie bialgebroid in Proposition \ref{prop:poisson-to-bialgebroid} is simply the De Rham differential operator.
 The following theorem is 
standard \cite{MR2911881, MR1262213, MR1746902}\footnote{In
literature, this theorem is  normally stated for global Lie groupoids,
 for which one needs to assume source connectedness and source simply
 connectedness.  The conclusion (as well as the proof) holds for local
 Lie groupoids without such topological assumptions.}, which extends a
well-known classical result of Drinfeld concerning Poisson
 Lie groups \cite{MR688240, MR934283}.

\begin{theorem}\mbox{}\label{thm:mackenzie-xu}
\begin{itemize}
\item[(i).] Lie bialgebroids $(A,A^*)$ are in one-to-one correspondence with Poisson local groupoids $(\rgpd\toto M,\bpi)$. 
\item[(ii).] The correspondence in (i) is functorial. More precisely,
 let $(A,A^*)$ and $(B,B^*)$ be Lie bialgebroids over $M$. Then 
morphisms $(\psi, \transpose{\psi}) : (A,A^*) \to (B,B^*)$ of Lie bialgebroids are in one-to-one correspondence with morphisms of the associated Poisson local groupoids 
  $$\bar\psi : (\rgpd_A\toto M, \bpi_A) \to (\rgpd_B\toto M, \bpi_B). $$
\item[(iii).] Let $(\rgpd\toto M, \bomega)$ be a symplectic local groupoid,
 and let $\bpi$ be the multiplicative Poisson bivector field  on $\Sigma$
inverse to $\bomega$. Then, as a Poisson groupoid, $(\rgpd\toto M, \bpi)$ has Lie bialgebroid $((T^*M)_\pi, TM)$. Here $TM$ is the tangent bundle Lie algebroid of $M$.
\end{itemize}
\end{theorem}

For completeness, let us recall that, by a morphism of local Lie groupoids,
we mean a smooth map $\bar\psi_1 : \rgpd \to \rgpd'$,
  defined on  neighborhoods of the unit spaces of $\rgpd$ and $\rgpd'$,
together with a smooth map  $\bar\psi_0 : M \to M'$ such that
\begin{equation*}
  \begin{gathered}
  \begin{tikzpicture}
    \matrix (m) [matrix of math nodes, row sep=3em, column sep=4em] {
      \rgpd & \rgpd'\\
      M & M' \\
    };
    \path[-stealth]
    (m-1-1) edge node [above] {$\bar\psi_1$} (m-1-2)
    (m-1-2) edge [transform canvas={xshift=-0.3em}]  (m-2-2)
    (m-1-2) edge [transform canvas={xshift=0.3em}]  (m-2-2)
    (m-1-1) edge [transform canvas={xshift=-0.3em}]  (m-2-1)
    (m-1-1) edge [transform canvas={xshift=0.3em}]  (m-2-1)
    (m-2-1) edge node [below] {$\bar\psi_0$} (m-2-2);
  \end{tikzpicture}
\end{gathered}
\end{equation*}
 satisfies the usual axioms of a Lie groupoid morphism.

\section{Poisson--Nijenhuis manifolds}\label{app:pn-manifolds}

Recall that a $(1,1)$-tensor $N : TM\to TM$ on a smooth
 manifold $M$ is called \emph{Nijenhuis} if its Nijenhuis torsion $T_N : \wedge^2 TM \to TM$ vanishes, where
 \begin{equation}\label{eqn:nijenhuis-torsion-def}
   T_N(X,Y) = [NX,NY] - N([NX,Y] + [X,NY]) + N^2[X,Y], \quad \forall X,Y \in \XX(M).
 \end{equation}

\begin{definition}\label{def:pn-manifold}
  Let $\pi$ be a Poisson bivector field
 on $M$ and $N$ be a Nijenhuis $(1,1)$-tensor. We say  that
the triple $(M,\pi,N)$ is a \emph{Poisson--Nijenhuis manifold} \cite{MR773513, MR1077465}
 if $\pi$ and $N$ satisfy the following compatibility relations for all 
 $\xi, \eta \in \Omega^1(M)$:
 \begin{align}
   N\smalcirc \pi^\sharp &= \pi^\sharp \smalcirc \transpose{N}, \label{eq:pn-compatibility-1}\\
   [\xi, \eta]_{\pi_N} &= [\transpose{N}\xi,\eta]_\pi + [\xi, \transpose{N}\eta]_\pi - \transpose{N} [\xi, \eta]_\pi. \label{eq:pn-compatibility-2}
\end{align}
Here $\pi_N$ is the bivector field on $M$
 defined by $\pi_N^\sharp =   N\smalcirc \pi^\sharp$
and $[\cdot,\cdot]_{\pi_N}$
 is the associated bracket on $\Omega^1(M)$.
\end{definition}

The following is standard in the theory of Poisson--Nijenhuis manifolds \cite{MR773513,MR1077465,MR1390832}.

\begin{proposition}\label{prop:second-poisson-pn}
  Let $(M,\pi,N)$ be a Poisson--Nijenhuis manifold. Then the bivector
field  $\pi_N \in \XX^2(M)$ defined by the property that
\begin{equation}
\label{eq:ATH}
 \pi_N^\sharp =   N\smalcirc \pi^\sharp
\end{equation} 
is a Poisson bivector field.
\end{proposition}

An alternative description of various compatibility relations between $\pi$ and $N$ is summarized in the following well-known result.

\begin{theorem}[\cite{MR1345608, MR1077465}]\label{thm:alt-pn-breakdown}
  Let $\pi\in \XX^2(M)$ be a Poisson bivector field on a manifold $M$ and let $N: TM \to TM$ be a $(1,1)$-tensor. Then the tensor $\pi_N$ defined by
  $$ \pi_N(\xi,\eta) = \eta \big( N \pi^\sharp \xi \big), 
 \quad\quad \forall \xi, \ \eta\in \Omega^1(M) $$
  is skew-symmetric if and only if \eq{eq:pn-compatibility-1} holds.
 In this case, we  also have the following assertions:
  \begin{itemize}
  \item[(i)] $[\pi,\pi_N] = 0$ if \eq{eq:pn-compatibility-2} holds, 
and the converse holds if $\pi$ is non-degenerate; 
  \item[(ii)] $[\pi_N, \pi_N]=0$  if $N$ is Nijenhuis.
  \end{itemize}
\end{theorem}


\begin{definition}
  A \emph{symplectic-Nijenhuis manifold} 
 is a Poisson--Nijenhuis manifold $(M,\pi,N)$ whose Poisson bivector field $\pi$ is non-degenerate.
\end{definition}
A symplectic-Nijenhuis manifold is also denoted by $(M,\omega,N)$,
where $\omega$ is the symplectic form corresponding to $\pi$.

The following theorem, due to Vaisman \cite{MR1345608}, essentially
asserts that  symplectic-Nijenhuis manifolds are equivalent to
biHamiltonian systems with one Poisson structure being
non-degenerate.

\begin{proposition}
\cite[Corollary 1.5]{MR1345608}
\label{pro:Vaisman}
Let $\pi$ and $\pi'$ be  compatible 
Poisson structures on a smooth manifold $M$, i.e.,
$$[\pi, \pi]=[\pi', \pi']=[\pi,  \pi']=0.$$
Assume that $\pi$ is non-degenerate. Then
$(M,\pi,N)$ is a symplectic-Nijenhuis manifold such that $\pi_N=\pi'$,
where $N$ is the $(1, 1)$-tensor on $M$ defined by
$N= (\pi')^\sharp \smalcirc  (\pi^\sharp)^{-1}: TM\to TM$.  
\end{proposition}

\begin{definition}\label{dfn:poisson-nijenhuis-map}
  Let $(X,\pi_X, N_X)$ and $(Y,\pi_Y, N_Y)$ be  Poisson--Nijenhuis manifolds. A \emph{Poisson--Nijenhuis map}
 is a smooth map $f: X\to Y$ such that 
  $$ f_* \smalcirc N_X = N_Y \smalcirc f_*, \quad\text{and}\quad f_* \pi_X = \pi_Y. $$
\end{definition}

If $f: X \to Y$ is a Poisson--Nijenhuis map, then $f_* \pi_{N_X} = \pi_{N_Y}$ as well. The following is easily seen.

\begin{proposition}\label{prop:pn-map}
  Let $(X,\pi=\pi_R + i\pi_I)$ and $(Y,\pi'=\pi'_R  +i\pi'_I)$ be 
holomorphic Poisson manifolds with almost complex structures $J_X$ and
 $J_Y$, respectively. Let $f : X \to Y$ be a smooth map. Then
\begin{itemize}
  \item[(i)] the map $f$ is  holomorphic Poisson  if and only if it
 is a Poisson--Nijenhuis map from $(X,\pi_I, J_X)$ to $(Y, \pi'_I, J_Y)$.
\item[(ii)] In particular, if $f$ is a holomorphic map,
 then $f$ is  holomorphic Poisson  if and only if $f_* \pi_I = \pi'_I$.
\end{itemize}
\end{proposition}

\bibliography{reference}
\bibliographystyle{plain}

\end{document}